\documentclass{article}

\usepackage{fullpage}

\usepackage{amsmath, amsthm, amssymb}
\usepackage{graphics}
\usepackage{epsfig}
\usepackage{color}

\def\ds{\displaystyle}
\def\Rset{{\rm I}\!{\rm R}}
\def\Nset{{\rm I}\!{\rm N}}
\def\ul{\underline}

\newtheorem{proposition}{Proposition}
\newtheorem{theorem}{Theorem}

\newtheorem{lemma}{Lemma}
\newtheorem{remark}{Remark}

\def\qed{\hfill \ensuremath{\Box}}

\begin{document}

\title{Global dynamics of the buffered chemostat\\ for a general class of
  response functions}
\author{A. Rapaport${}^{1,*}$, I. Haidar${}^{2}$ and
  J. Harmand${}^{3,*}$\\
{\small ${}^{1}$ MISTEA, UMR 729 INRA-SupAgro, Montpellier, France}\\
{\small ${}^{2}$ LSS-Supelec, Gif-sur-Yvette, France}\\
{\small ${}^{3}$ LBE, INRA, Narbonne, France}\\
{\small ${}^{*}$ MODEMIC, INRIA Sophia-Antipolis M\'editerran\'ee, France}}
\date{\today}

\maketitle

\begin{abstract}
We study how a particular spatial structure with a buffer
impacts the number of equilibria and their stability in
the chemostat model. 
We show that the occurrence of a buffer can allow a species to persist or on the
opposite to go extinct, depending on the characteristics of the buffer.
For non-monotonic response functions, we characterize the buffered
configurations that make the chemostat dynamics globally
asymptotically stable, while this is not possible with single, serial
or parallel vessels of the same total volume and input flow.
These results are illustrated with the Haldane kinetic function.\\



\end{abstract}
\section{Introduction}
The chemostat was introduced in the
fifties as an experimental
device to study the microbial growth on a limiting
resource \cite{M50,NS50}.
It is also often used as a mean to reproduce situations where 
(limiting) nutrients are fed to micro-organisms, typically in 
a liquid medium, in natural ecosystems \cite{HJ54,DGC02} or
anthropized environments \cite{L77}. More generally, the chemostat
is largely used as a scientific investigation tool in microbial ecology 
\cite{JM74,V77}.

The mathematical model of the chemostat has been
extensively studied (see e.g. \cite{SW95}) and used as a reference
model in microbiology \cite{P75}, microbial ecology \cite{FS81}
or biotechnological industries such as the waste-water treatment
\cite{DV01}.
More generally, the chemostat
serves to describe resource-consumer relations, where the resource is supplied at a constant rate.
However, in many applications, the assumption of perfectly stirred
chemostats is, in general, too restrictive. In the eighties, the gradostat, as an
experimental device composed of a 
set of chemostats of identical volume interconnected in series, was
introduced to represent spatial gradient \cite{LW81}, in a 
marine environment \cite{HYSS98} or to model rhizosphere
\cite{FHN11}. It motivated
several mathematical studies
\cite{T86,JSTW87,EE90,ST89,Z92,SW91,HS94,SW00,GTVP09}. 
Similarly, an interest for series of bioreactors appeared in
biochemical industry, with tanks of different volumes to be
minimized \cite{LT82,HR89,DBBT96,HRT99,DHRL06}.
In ecology, island models have been proposed since the late
sixties \cite{MW67} to study the effects of heterogeneous
environments with more general patterns than serial ones.
Several studies of prey-predator in patchy environments have been conducted since then \cite{L74,AN01}.
Comparatively, relatively few studies have considered non-serial
interconnections for resource-consumer models or chemostats \cite{STW91}.
In natural reservoirs such as in undergrounds or ground-waters, 
a spatial structure with interconnections between 
several volumes is often considered, each of them being approximated
as perfectly mixed tank. 
Those interconnections can be parallel, series or built up in more complex networks. To our knowledge, the 
influence of the topology of a network of chemostats on the overall
dynamics has been sparsely investigated in the literature.
However, the simple consideration of two interconnected habitats can lead to
non-intuitive behaviors \cite{SF79,NT06,RHM08,LTVP98} and influence
significantly the overall performances
\cite{NS06,HRG11}.
Recently, literature in ecology has raised the relevance of
``source-sink'' models for describing plants/nutrients
interactions, and predicting ecosystems performances \cite{L10,GGLM10,LDGGGLLMM13}. 
Those models are mathematically close to general
gradostat models, but with a significant difference concerning the
resources compartments, for which the input rate mechanisms (due to atmospheric
depositions or rock alterations) are assumed to be independent of the
nutrient leaching (and not modeled as a transport term as in hydrology
or in chemostat-like models). 

It is also well-known since the seventies that microbial growth can
be inhibited by large concentrations of nutrient. Such inhibition can be
modeled by non-monotonic response functions \cite{A68,BC76} and lead 
to initial-condition dependent washout \cite{BW85,WL92,L98}.
Non-monotonic response functions occur in predator-prey models,
for instance, when the
predation decreases due to the ability of the prey to better defend
when their population get larger. 
This non-monotonic functional response could also lead to bi-stability and possible extinction of the
predator \cite{FW86,XR01}.

Several control strategies of the input flow were proposed in the 
literature to globally stabilize the chemostat \cite{DB84,HRM06,RH08,SAL12} but the ability of a
spatial structure to passively stabilize such dynamics has not been
yet studied (in \cite{STW91} a general structure of networks of chemostats
is considered but with monotonic growth rates, while in \cite{T94} 
non-monotonic functions are considered but for the serial gradostat only).

The present work considers a particular interconnection of two
chemostats of different volumes, one being a buffer tank.
To our knowledge, this spatial structure, that is neither serial nor
parallel, has not yet been considered in the literature.
This structure is analogous to
refuges in patchy environments \cite{AN01}, but here both consumer and resource
are present in each vessel.
We prove that it is possible with such a configuration
to obtain repulsive washout equilibrium,
while any serial, parallel or single tank structures with the
same total volume exhibits multi-stability.
This result brings new insights into the role of spatial patterns in the stability of bio-conversion
processes in natural environments, where buffers can occur
such as in soil ecosystems. It has also
potential implications for the design of robust industrial bio-processes.

The paper is organized as follows. Section \ref{hypo} presents the
 hypotheses and the buffered configuration, comparing with
serial and parallel interconnections.
Section \ref{section3} studies the multiplicity of equilibria and
their stability for such configurations, considering a general class of response functions (monotonic as
well as non-monotonic).
Section \ref{section4} discusses the biological and ecological
implications of the results of Section \ref{section3} in terms of
persistence of microbial species in a non-homogeneous environment, along
with some industrial perspectives.
Numerical simulations illustrate the results on an Haldane
function in Section \ref{section_example}.
All the proofs are postponed to the Appendix.

\section{General considerations}
\label{hypo}
We consider the chemostat model with a single 
strain growing on a single limiting nutrient.
The system is fed with nutrient of concentration $S_{in}$ 
with flow rate $Q$.
The total volume $V$ is assumed to be constant
(i.e. input and output flow rates are supposed to be identical).
When the concentrations of nutrient (or substrate) and biomass, denoted
respectively $S$ and $X$,
are homogeneous, as it is the case in perfectly mixed tanks, 
the system can be modeled by the well-known differential equations:
\begin{equation}
\label{chemostat0}
\begin{array}{lll}
\dot S & = & \ds -\frac{\mu(S)}{Y}X+\frac{Q}{V}(S_{in}-S) \ ,\\
\dot X & = & \ds \mu(S)X-\frac{Q}{V}X \ ,
\end{array}
\end{equation}
where $\mu(\cdot)$ is the uptake function and $Y$ the yield
coefficient of the transformation of nutrient into biomass. 
Without any loss of generality, we take $Y=1$ (at the price of
changing $X$ in $YX$). For convenience, we define the dilution rate
\[
D=\frac{Q}{V} \ .
\]
We consider quite general uptake functions, that fulfill the
following properties.\\

\noindent {\bf Assumptions A1.}
\begin{itemize}
\item[i] The function $\mu(\cdot)$ is analytic and such that
$\mu(0)=0$, $\mu(S)>0$ for any $S>0$.
\item[ii] The function $\mu(\cdot)$ is either increasing, or there exists $\hat
S>0$ such that $\mu(\cdot)$ is increasing on $(0,\hat S)$ and
decreasing on $(\hat S,+\infty)$.
\end{itemize}

\medskip

The usual uptake functions, such as the Monod function \cite{M50}
\begin{equation}
\mu(S)=\frac{\mu_{\max}S}{K_{s}+S} \ ,
\end{equation}
or the Haldane one \cite{A68}
\begin{equation}
\label{Haldane}
\mu(S)=\frac{\bar\mu S}{K+S+S^{2}/K_{I}} \ ,
\end{equation}
fulfill theses hypotheses.
Classically, we consider the set
\begin{equation}
\label{Lambda}
\Lambda(D)=\{ S >0 \; \vert \; \mu(S) > D\}
\end{equation}
that plays an important role in the determination of the equilibria 
of the system. 
Under Assumptions A1, the set $\Lambda(D)$ is either empty or 
an open interval that we denote 
\[
\Lambda(D)=(\lambda_{-}(D),\lambda_{+}(D)) \ ,
\]
where $\lambda_{+}(D)$ can be equal to $+\infty$.\\

We recall from the theory of the chemostat model
(see for instance \cite{SW95}) that under Assumptions A1 there are three kinds of phase
portrait of the dynamics (\ref{chemostat0}), depending on the parameter $S_{in}$.

\begin{proposition}
\label{prop-chemostat}
Assume that Hypotheses A1 are fulfilled.
\begin{itemize}
\item[-] {\em Case 1:} $\Lambda(D)=\emptyset$ or $\lambda_{-}(D)\geq S_{in}$.
  The washout equilibrium $E_{0}=(S_{in},0)$ is the unique non
  negative equilibrium of system (\ref{chemostat0}). 
  Furthermore it is globally attracting.
\item[-] {\em Case 2:} $S_{in} > \lambda_{+}(D)$. The system (\ref{chemostat0}) has three non-negative
  equilibria $E_{-}(D)=(\lambda_{-}(D),S_{in}-\lambda_{-}(D))$,
  $E_{+}(D)=(\lambda_{+}(D),S_{in}-\lambda_{+}(D))$ and  $E_{0}=(S_{in},0)$.
  Only $E_{-}(D)$ and $E_{0}$
  are attracting, and the dynamics is bi-stable.
\item[-] {\em Case 3:} $S_{in} \in \Lambda(D)$. The system (\ref{chemostat0}) has two non negative equilibria
$E_{-}(D)=(\lambda_{-}(D),S_{in}-\lambda_{-}(D))$ and $E_{0}=(S_{in},0)$. $E_{-}(D)$ is globally attracting on the positive quadrant.
\end{itemize}
\end{proposition}

\medskip

Notice that in case 2, the qualitative behavior of the growth can change radically depending on the initial condition.\\

The question we investigate in this paper is related to the assumption
that the vessel is perfectly mixed, and to the role that a spatial structure 
could have on the stability of the dynamics.
Consider the case for which the washout equilibrium is attracting in the chemostat
model (cases 1 and 2 of Proposition
\ref{prop-chemostat}). Furthermore, consider
spatial configurations with the same input flow and
residence time than the perfectly mixed case, i.e. with the same total
volume $V$ and input flow $Q$. Then, one has the following property. 

\begin{lemma}
\label{LemmaSerialParallel}
Assume that Hypotheses A1 are fulfilled and let $Q$ and $V$ be such that $S_{in}\notin\Lambda(D)$.
Then the washout is an attracting equilibrium in at least one vessel of any
interconnection in series or in parallel of $n$ tanks of volume
$V_{i}$ such that $\sum_{i=1}^{n}V_{i}=V$, assuming that each of them
is perfectly mixed.
\end{lemma}

This Lemma shows that when a bacterial species cannot persist in a
chemostat, from any or a subset of initial conditions, this property
persists in at least one vessel of any serial or parallel
interconnection of chemostats with the same total volume.
In the present work, we study a different kind of
spatial configuration with an asymmetry created by two interconnected
volumes, one of them serving as a buffer (see Figure
\ref{fig-pocket}). We call these spatial configurations a ``buffered chemostat'',
to be compared with the ``single chemostat''.
\begin{figure}[ht]
\begin{center}
\includegraphics[width=4cm]{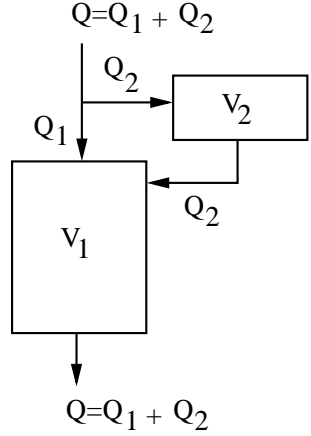}
\caption{
The buffered chemostat.
\label{fig-pocket}}
\end{center}
\end{figure}
$V_{1}$ and $V_{2}$ are respectively the volumes of the main tank and
the buffer, and $Q_{1}$ and $Q_{2}$ denote the
input flow rates of each tank, with $Q=Q_{1}+Q_{2}$. 
We assume that each vessel is perfectly mixed.
Straightforwardly, the dynamical equations of the buffered chemostat are
\begin{equation}
\label{chemostat2}
\begin{array}{lll}
\dot S_{1} & = & \ds
-\mu(S_{1})X_{1}+\frac{Q_{1}S_{in}+Q_{2}S_{2}-QS_{1}}{V_{1}} \ ,\\
\dot X_{1} & = & \ds \mu(S_{1})X_{1}+\frac{Q_{2}X_{2}-QX_{1}}{V_{1}} \
,\\
\dot S_{2} & = & \ds
-\mu(S_{2})X_{2}+\frac{Q_{2}S_{in}-Q_{2}S_{2}}{V_{2}} \ ,\\
\dot X_{2} & = & \ds \mu(S_{2})X_{2}-\frac{Q_{2}X_{2}}{V_{2}} \ .
\end{array}
\end{equation}
Notice that the limiting case $V_{1}=0$ consists in a by-pass of the
volume $V_{2}$ with a flow $Q_{1}$.\\

In the next Section, we study the equilibria of this model, their
multiplicity and their stability.

\section{Analysis of  the dynamics of the buffered chemostat}
\label{section3}

Given a volume $V$ and an input flow rate $Q$, we
describe the set of all possible buffered configurations with
$Q=Q_{1}+Q_{2}$ and $V=V_{1}+V_{2}$ by two parameters $r \in (0,1)$ and
$\alpha> 0$
defined as follows 
\[
r=\frac{V_{1}}{V}, \quad \alpha=\frac{Q_{2}}{(1-r)Q} \ .
\]
This choice of parameterization is more convenient than the original
one because it decouples more easily the role of the two
parameters, as it is shown by equations (\ref{chemostat2b}) below.

Dynamics (\ref{chemostat2}) can then be written in the following way
\begin{equation}
\label{chemostat2b}
\begin{array}{lll}
\dot S_{1} & = & \ds
-\mu(S_{1})X_{1}+D\frac{\alpha(1-r)(S_{2}-S_{1})+
(1-\alpha(1-r))(S_{in}-S_{1})}{r} \ ,\\[2mm]
\dot X_{1} & = & \ds \mu(S_{1})X_{1}+D\frac{\alpha(1-r)(X_{2}-X_{1})-
(1-\alpha(1-r))X_{1}}{r} \ ,\\[2mm]
\dot S_{2} & = & \ds -\mu(S_{2})X_{2}+D\alpha(S_{in}-S_{2}) \ ,\\[2mm]
\dot X_{2} & = & \ds \mu(S_{2})X_{2}-D\alpha X_{2} \ .
\end{array}
\end{equation}
At equilibrium, one should have $\dot S_{2}+\dot X_{2}=\alpha
D(S_{in}-S_{2}-X_{2})=0$ that is $S_{2}+X_{2}=S_{in}$. Then, one
should have $\dot S_{1}+\dot X_{1}=D(S_{in}-S_{1}-X_{1})/r=0$ that is
$S_{1}+X_{1}=S_{in}$. Thus, equilibria
$(S_{1}^{\star},X_{1}^{\star},S_{2}^{\star},X_{2}^{\star})$
of dynamics (\ref{chemostat2b}) 
can be written as solutions of the following equations:
\begin{eqnarray}
\label{eq1}
1+\frac{1-r}{r}\left(1-\alpha\frac{S_{in}-S_{2}^{\star}}{S_{in}-S_{1}^{\star}}\right)=\frac{\mu(S_{1}^{\star})}{D}
\mbox{ or } \left\{ S_{1}^{\star}=S_{in} \mbox{ when } S_{2}^{\star}=S_{in}\right\}\ ,\\
\label{eq2}
X_{1}^{\star}=S_{in}-S_{1}^{\star} \ ,\\
\label{eq3}
\alpha=\frac{\mu(S_{2}^{\star})}{D} \mbox{ or } S_{2}^{\star}=S_{in} \ ,\\
\label{eq4}
X_{2}^{\star}=S_{in}-S_{2}^{\star} \ .
\end{eqnarray}
Due to the cascade structure of model (\ref{chemostat2}),
the study of the dynamics of the second reactor can be done 
independently of the first one. Depending of the value of $\alpha$,
the three cases given in Proposition \ref{prop-chemostat} for the single
chemostat are possible in the second tank. This implies the following
two possibilities for the equilibria of the first sub-system.

\begin{enumerate}
\item When $(S_{2}(\cdot),X_{2}(\cdot))$ converges to the washout equilibrium
  (cases 1 and 2), the $(S_{1},X_{1})$ dynamics is asymptotically
  equivalent to a single chemostat model with dilution rate $D/r$, and
  Proposition \ref{prop-chemostat} applies.

\item When  $(S_{2}(\cdot),X_{2}(\cdot))$ converges towards a positive
  equilibrium $(S_{2}^{\star}(\alpha),S_{in}-S_{2}^{\star}(\alpha))$ (cases 2
  and 3), we consider the family of hyperbola $H_{\alpha,r}$ that are the graphs of the functions
\begin{equation}
\label{phi}
\phi_{\alpha,r}(s)=1+\frac{1-r}{r}\left(1-\alpha\frac{S_{in}-S_{2}^{\star}(\alpha)}{S_{in}-s}\right)
\end{equation}
parameterized by $\alpha$ and $r\in (0,1)$.
From equations (\ref{eq1}) and (\ref{eq2}), a positive
equilibrium $(S_{1}^{\star},X_{1}^{\star})$ of (\ref{chemostat2b}) satisfies
\begin{equation}
\label{geometric_condition}
\phi_{\alpha,r}(S_{1}^{\star})=\mu(S_{1}^{\star})/D
\end{equation}
or equivalently $S_{1}^{\star}$ is the abscissa of an intersection of the graph of
$\mu(\cdot)/D$ with the hyperbola $H_{\alpha,r}$. Then, from  equation
(\ref{eq2}), to each solution
$S_{1}^{\star}$ corresponds a unique
$X_{1}^{\star}=S_{in}-S_{1}^{\star}$. Notice that the washout is not an
equilibrium for the first tank.
\end{enumerate}

In the following, we consider only non-trivial cases for which the second tank
admits a positive equilibrium, assuming the hypotheses:\\

\noindent {\bf Assumptions A2.}
Under Assumptions A1, $D$ and $\alpha$ are positive numbers such that
$\Lambda(\alpha D)\neq \emptyset$ and $\lambda_{-}(\alpha D) < S_{in}$.\\

Similar to the single chemostat that considers the set
$\Lambda(D)$ given in (\ref{Lambda}), we define the set
\begin{equation}
\label{Gamma}
\Gamma_{\alpha,r}(D)=\left\{ S \in (0,S_{in}) \, \vert \,
\mu(S)>D\phi_{\alpha,r}(S) \right\} \ .
\end{equation}

We shall also consider the subset of configurations for which
system (\ref{chemostat2b}) admits an unique positive equilibrium, denoted by
\begin{equation}
\label{defRbar}
\overline{R}_{\alpha}(D) = \left\{ r \in (0,1) \, \vert \,
\exists!\; s \in (0,S_{in}) \mbox{ s.t. } D\phi_{\alpha,r}(s)=\mu(s)
\right\} \ .
\end{equation}

\medskip

We state now our main results.

\begin{theorem}
\label{main_theo}
Assume that Hypotheses A1 and A2 are fulfilled. 
The set $\Gamma_{\alpha,r}(D)$ is non-empty, and for almost any $r \in
(0,1)$ one has the following properties, except from a subset of
initial conditions of zero Lebesgue measure.
\begin{itemize}
\item[i.] When the initial condition of the $(S_{2},X_{2})$ sub-system
  belongs to the attraction basin of $(S_{in},0)$, the solution $(S_{1},X_{1})$ 
  of system (\ref{chemostat2b}) converges
  exponentially to the rest point
  $(\lambda_{-}(D/r),S_{in}-\lambda_{-}(D/r))$ when
  $\lambda_{-}(D/r)<S_{in}$, or to the washout equilibrium when
  $\mu(S_{in})<D/r$.
\item[ii.] When the initial condition of the $(S_{2},X_{2})$ sub-system
  does not belong to the attraction basin of $(S_{in},0)$,
the trajectory of the system (\ref{chemostat2b}) converges
exponentially to a positive equilibrium\\
$(S_{1}^{\star},S_{in}-S_{1}^{\star},\lambda_{-}(\alpha D),S_{in}-\lambda_{-}(\alpha D))$ where $S_{1}^{\star}$ is the left endpoint of a connected
component of $\Gamma_{\alpha,r}(D)$.
\end{itemize}
Moreover, the set $\overline{R}_{\alpha}(D)$ is non-empty.
\end{theorem}

Let give some observations on these results.
\begin{itemize}
\item[-] In contrast to the single chemostat, for which the set $\Lambda(D)$ could
be empty, the set
$\Gamma_{\alpha,r}(D)$ is non-empty.  This means that dynamics 
(\ref{chemostat2b}) always admits a positive equilibrium, even when
the washout is the only equilibrium of the single chemostat, contrary
to serial or parallel chemostats (cf Lemma \ref{LemmaSerialParallel}).

\item[-] When the initial condition of the $(S_{2},X_{2})$ sub-system belongs
to the attraction basin of $(S_{in},0)$ (that could be reduced to a
singleton), it is a not a surprise that the asymptotic behavior of the sub-system $(S_{1},X_{1})$
is the same as for a single chemostat with a dilution rate equal to
$D/r$ (cf point i.). Otherwise, the whole state converges to a positive
equilibrium, with a possible multiplicity of equilibria (cf point
ii.). Here, a remarkable feature
is the existence of buffered configurations $(\alpha,r)$ that
possess an unique globally asymptotically stable equilibrium
(when $\alpha D<\mu(S_{in})$ and $r \in \overline{R}_{\alpha}(D)$), in
contrast to the single chemostat or any serial or parallel
configurations for which a bi-stability occur
when the functional response is non-monotonic. 
\end{itemize}

To help grasp the geometric condition (\ref{geometric_condition}) that
is the key for the characterization of the equilibria,
we introduce the number
\begin{equation}
\label{ulS}
\ul S(\alpha) = \alpha S_{2}^{\star}(\alpha) +(1-\alpha)S_{in} \ ,
\end{equation}
that fulfills the remarkable property
\[
\phi_{\alpha,r}(\ul S(\alpha))=1, \qquad \forall r \in (0,1) \ .
\]
We first explicit the condition (\ref{geometric_condition}) on the
specific case of the Haldane function (\ref{Haldane}):
\begin{equation}
\label{polynom}
D(S_{in}-s-\alpha(1-r)(S_{in}-S_{2}^{\star}(\alpha))(K+s+s^{2}/K_{I})=r\bar\mu
s(S_{in}-s) \ .
\end{equation}
$S_{1}^{\star}$ is then a root of a polynomial $P$ of degree 3.
So there exist at most three solutions of
$\phi_{\alpha,r}(s)=\mu(s)/D$. 
For small values of $r$, we remark that
$\phi_{\alpha,r}(0)$ is very large and $\phi_{\alpha,r}$ has a high
slope. On the contrary, for $r$ near to $1$, $\phi_{\alpha,r}(0)$ is
closed to $1$ and $\phi_{\alpha,r}$ has a light slope. Intuitively,
we expect to have only one root for small values of $r$ and three
for large values of $r$.
For $\bar r$ such that there
exists a solution $S_{1}^{\star}$ of $\phi_{\alpha,\bar r}(s)=\mu(s)/D$
and $\phi_{\alpha,\bar r}^{\prime}(s)=\mu^{\prime}(s)$, one has
$P(S_{1}^{\star})=0$ and $P^{\prime}(S_{1}^{\star})=0$, that is
$S_{1}^{\star}$ is a double root of $P$ (and there exists at most
one such double root because $P$ is of degree 3). At such $S_{1}^{\star}$,
the hyperbola $H_{\alpha,\bar r}$ is tangent to the graph of $\mu(\cdot)$.
Intuitively, this corresponds to the limiting case for the parameter $r$ in
between cases for which there is one or three roots (see Figures
\ref{fig-phi} and \ref{fig-phi2} where tangent hyperbola are drawn in
thick line).
\begin{figure}[h]
\begin{center}
\includegraphics[width=5cm]{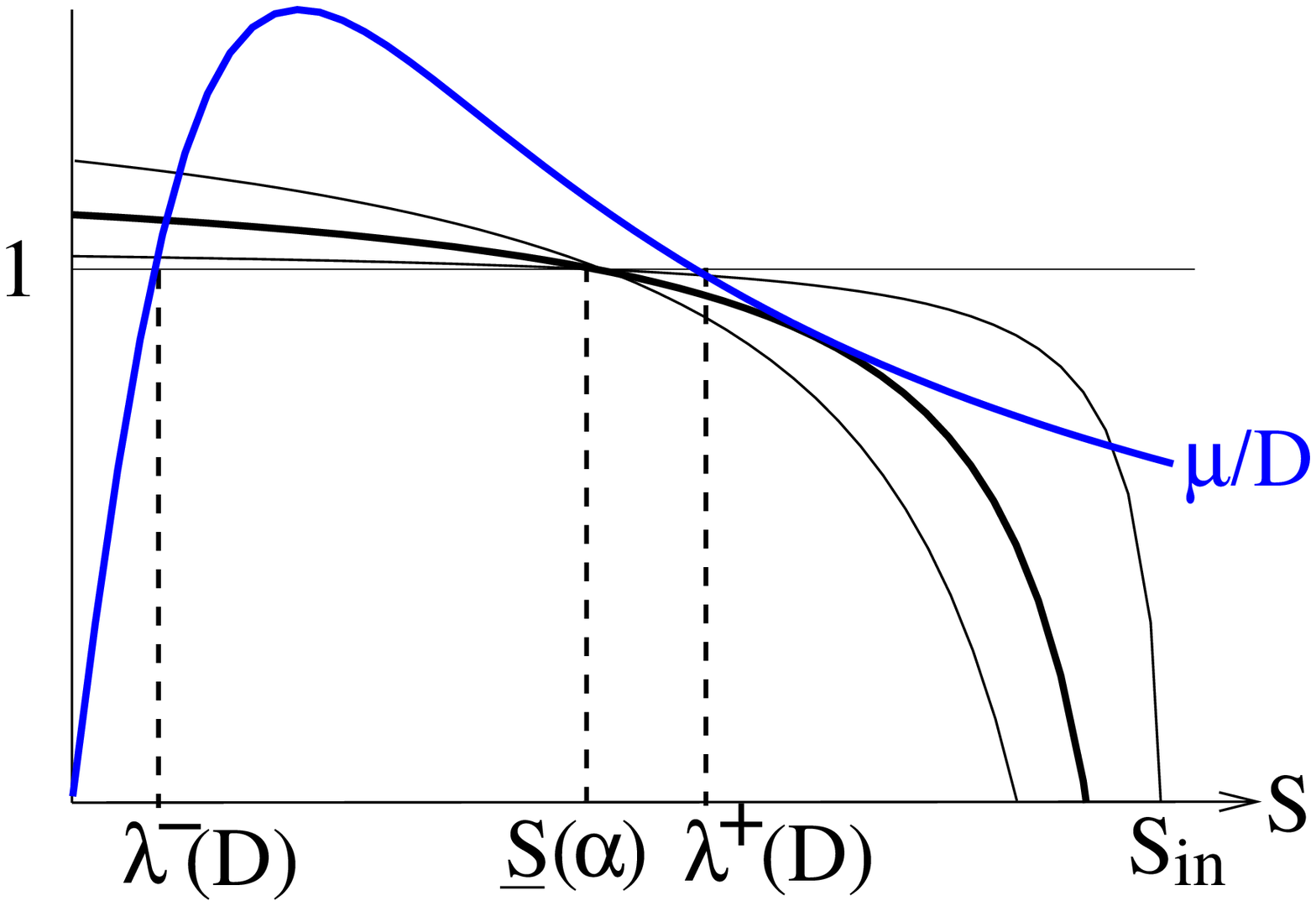} \hspace{2mm}
\includegraphics[width=5cm]{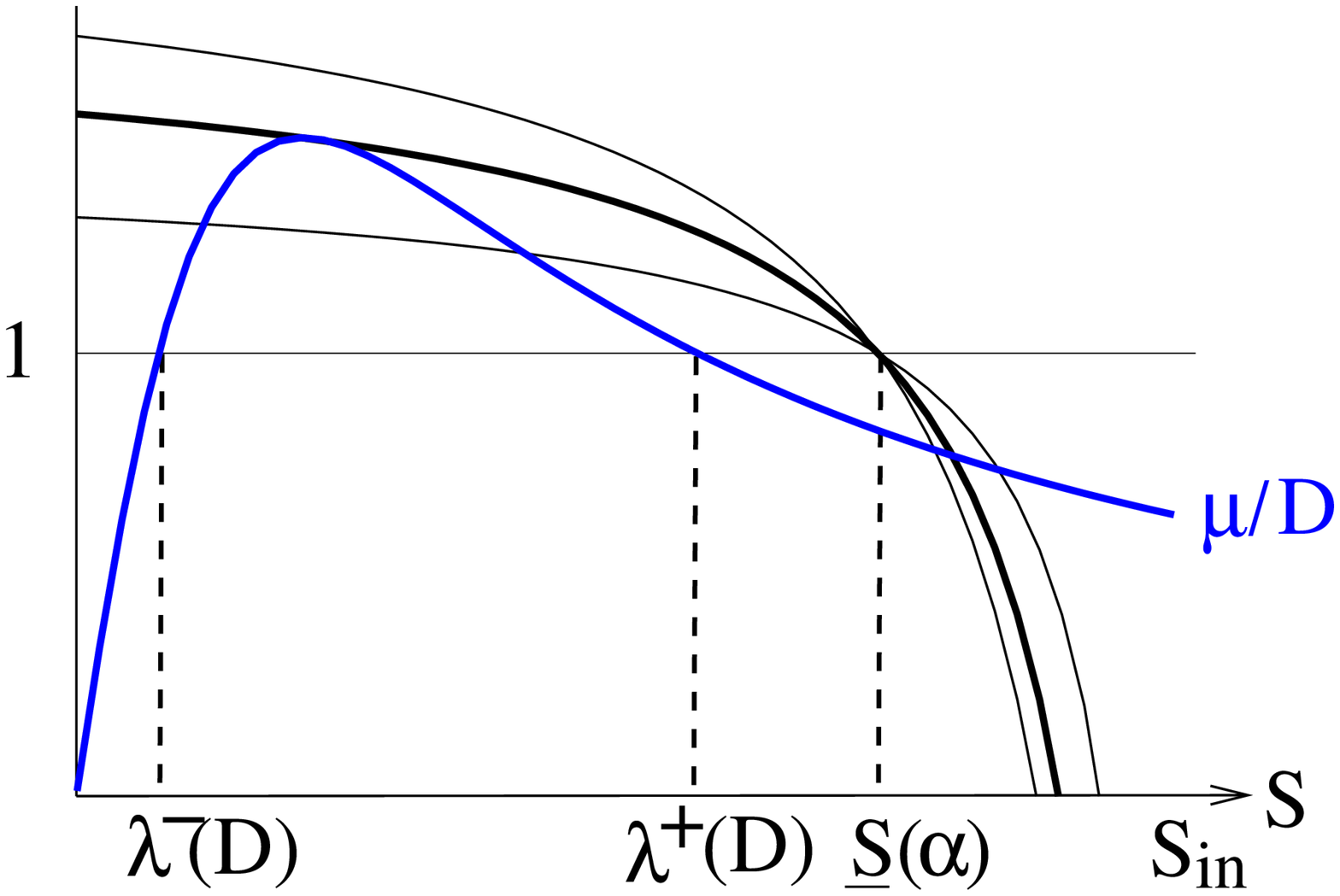}
\caption{Subset of functions $\phi_{\alpha,r}(\cdot)$ when $\ul
  S(\alpha)<\lambda_{+}(D)$ (on the left) and $\ul S(\alpha)>\lambda_{+}(D)$
  (on the right), illustrated with an Haldane function (when
  $\lambda_{+}(D)<S_{in})$) [parameters: $\bar\mu=12$, $K=1$,
  $K_{I}=0.1$, $S_{in}=2$, $D=1.1$, $\alpha=0.64$ (left) / $0.36$ (right)].
\label{fig-phi}}
\end{center}
\end{figure}

\begin{figure}[h]
\begin{center}
\includegraphics[width=5cm]{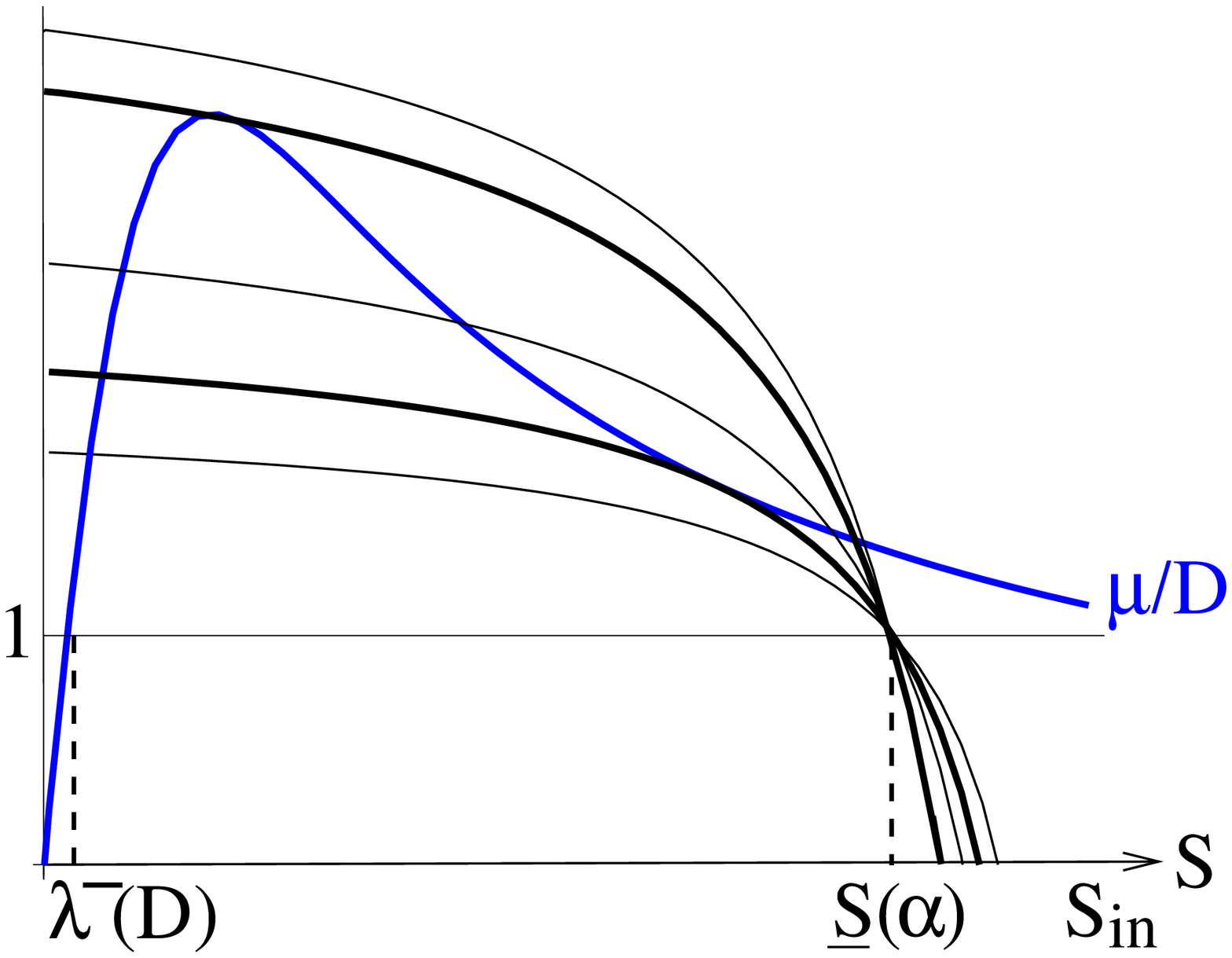} \hspace{2mm}
\includegraphics[width=5cm]{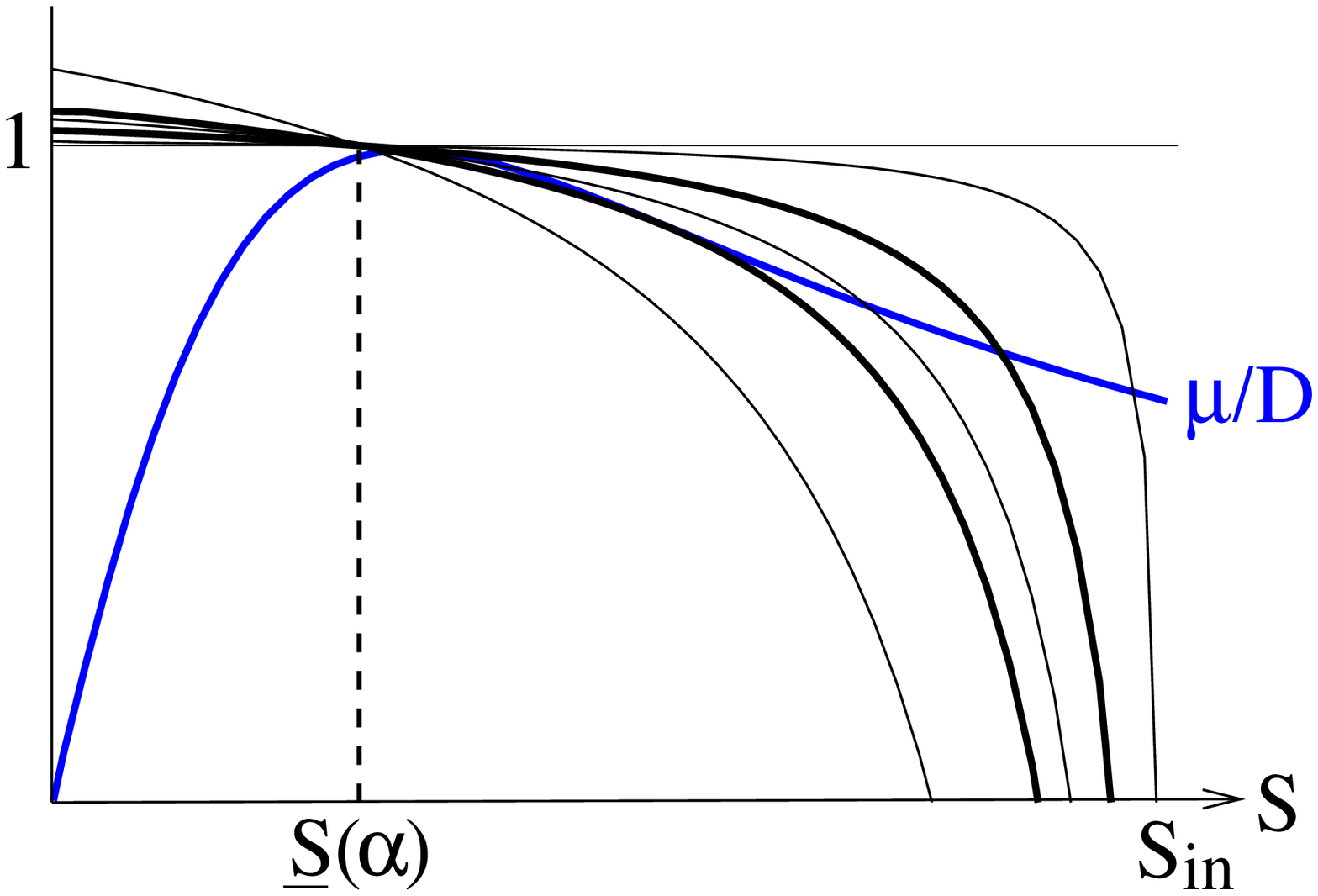}
\caption{Subset of functions $\phi_{\alpha,r}(\cdot)$ illustrated with
  an Haldane function when $\lambda_{-}(D)<S_{in}<\lambda_{+}(D)$ (on
  the left) and when $\Lambda(D)=\emptyset$ (on the right)
[parameters: $\bar\mu=12$, $K=1$,
  $K_{I}=0.1$, $S_{in}=2$ (left) /$1$ (right), $D=0.5$ (left) /
  $1.65$ (right), $\alpha=0.2$ (left) / $0.9$ (right)].
\label{fig-phi2}}
\end{center}
\end{figure}

\medskip

To formalize these observations for more general growth
functions $\mu(\cdot)$ that fulfill Assumptions A1, we consider the
set of $s$ at which the hyperbola $H_{\alpha,r}$ is tangent to the graph of the function
$\mu(\cdot)/D$ and is locally on one side
(that amounts to have $0$ as a local extremum of the function
$\phi_{\alpha,r}(\cdot)-\mu(\cdot)/D$ at $s$):
\begin{equation}
\label{calS}
{\cal S}_{\alpha,r}(D)=\left\{ s \in (0,S_{in}) \mbox{ s.t. }
\min\left\{ n \in \Nset \, \vert \, D\frac{d^{n}\phi_{\alpha,r}}{ds^{n}}(s)\neq 
\frac{d^{n}\mu}{ds^{n}}(s) \right\} \mbox{ is even and larger than 1} \right\}
\end{equation}
along with the set
\begin{equation}
\label{defRalpha}
R_{\alpha}(D)=\left\{ r \in (0,1) \mbox{ s.t. } {\cal S}_{\alpha,r}(D)
\neq \emptyset \right\} \ .
\end{equation}

One can distinguish two cases:
\begin{enumerate}
\item {\em The single chemostat has only one
attracting equilibrium.} Tangencies  of the graphs of $\phi_{\alpha,r}$ and $\mu$
could occur for certain values of $r$ (see Figure \ref{fig-phi2} as an
illustration), leading to non-empty set $R_{\alpha}(D)$ and
multi-equilibria.
Another remarkable feature is that the buffer could create a multiplicity of equilibria.

\item {\em The single chemostat presents a bi-stability.}
The function $\mu$ is necessarily non-monotonic on $(0,S_{in})$
and a tangency of the graphs of $\phi_{\alpha,r}$ and $\mu$ always
occurs for a certain $r$ with an abscissa that is located
\begin{itemize}
\item[-] either at the right of $\lambda_{+}$ when $\ul
  S(\alpha)<\lambda_{+}(D)$ (see the right picture of Figure \ref{fig-phi}),
\item[-] either at the left of $\lambda_{+}$ when $\ul
  S(\alpha)>\lambda_{+}(D)$ (see the left picture of Figure \ref{fig-phi}).
\end{itemize}

\end{enumerate}

In the Appendix, more properties on the sets $R_{\alpha}(D)$ and
the multiplicity of equilibria are given in the Proposition
\ref{prop1}.

\medskip

\begin{remark}
\label{remark1}
Under the conditions of Theorem \ref{main_theo}, consider the number
\begin{equation}
\label{bar_r}
\bar r_{D}(\alpha)=\sup \overline{R}_{\alpha}(D)
\end{equation}
that guarantees that for any $(r,\alpha)$ with $r<\bar
r_{D}(\alpha)$, the buffered chemostat model admits a unique (globally
asymptotically stable) positive equilibrium.

The map $(\alpha,r)\mapsto S_{1}^{\star}(\alpha,r)$, where
$S_{1}^{\star}(\alpha,r)$ is the unique solution of
(\ref{geometric_condition}) on $(0,S_{in})$, is clearly
continuous and one can then consider the limiting map:
\[
\bar S_{1}^{\star}(\alpha)=\lim_{r<\bar r_{D} (\alpha), \, r\to \bar
  r_{D}(\alpha)} S_{1}^{\star}(\alpha,r) \ .
\]
When $\lambda_{+}(D)<S_{in}$, one has
$\bar S_{1}^{\star}(\alpha) \leq \lambda_{+}(D)$  (resp. $\bar
S_{1}^{\star}(\alpha) \geq \lambda_{+}(D)$) when $\ul
S(\alpha)<\lambda_{+}(D)$ (resp. $\ul S(\alpha)>\lambda_{+}(D)$).
Consider, if it exists, a value of $\alpha$, denoted by $\ul \alpha$,
that is such that $\ul S(\ul \alpha)=\lambda_{+}(D)$. 
Although one has
$\phi_{\ul \alpha,r}(\lambda_{+}(D))=\mu(\lambda_{+}(D))/D$ for any $r$, 
there is no reason to have
\[
\lim_{\alpha<\ul \alpha, \, \alpha \to \ul \alpha}\bar S_{1}^{\star}(\alpha)=\lambda_{+}(D) \; 
\mbox{ or } \;
\lim_{\alpha>\ul \alpha, \, \alpha \to \ul \alpha}\bar S_{1}^{\star}(\alpha)=\lambda_{+}(D) \ .
\]
Consequently, the map $\alpha \mapsto \bar r_{D}(\alpha)$ might be discontinuous at
such point $\ul \alpha$.
In Section \ref{section_example}, the non-continuity of
the map $\alpha \mapsto \bar r_{D}(\alpha)$ is illustrated on the Haldane
function.
\end{remark}

\section{Discussion and comparison with the single chemostat}
\label{section4}

In this Section, we discuss the applications of Theorem \ref{main_theo}
in terms of ecological and biotechnological implications for different
buffered configurations.

\subsection{From an ecological point of view}
To better grasp the difference brought by a buffered spatialization
compared to a perfectly-mixed environment, we distinguish two
main cases depending on the washout if it is an attracting equilibrium or not in the
single chemostat.

\subsubsection{Washout is attracting in a single chemostat}
\label{subsection_attracting_washout}
Such situation corresponds to Cases 1 or 2 of Proposition
\ref{prop-chemostat}~: 
\begin{itemize}
\item[-] either the washout is the only equilibrium (and is necessarily
attracting). This happens when the dilution rate $D$ is too high or the
input concentration $S_{in}$ too low, that is when one has  
$D>\mu(S)$ for any $S \in [0,S_{in}]$,
\item[-] either the growth function $\mu(\cdot)$ is non-monotonic on
$(0,S_{in})$ with an non-empty set $\Lambda(D)$ such that
$\lambda_{+}(D)<S_{in}$. The system admits then two attracting
equilibria: a positive one and the washout.
\end{itemize}
For both cases, Theorem \ref{main_theo} shows that there exist buffered
configurations $(\alpha,r)$ (with $S_{in} \in (\alpha D)$ and
$r\in \overline{R}_{\alpha}(D)$) such that the overall dynamics has an
unique globally stable positive equilibrium. 
Recall, from Lemma \ref{LemmaSerialParallel}, that any species cannot
persist in both tanks with a serial or parallel configuration of
the same total volume, differently to the buffered interconnection.
This property demonstrates that a simple
(but particular) spatial structure such as the buffered one can
explain the persistence of a species in an environment that is
unfavorable if it was homogeneous.

Furthermore,
Theorem \ref{main_theo} shows that in absence of initial biomass in the main
tank, a species seeded in the buffer can invade and persist in
the main tank. We conclude that a buffer can play the role of a refuge.

\subsubsection{Single chemostat has a unique positive equilibrium}

We are in the conditions of Case 3 of Proposition
\ref{prop-chemostat}. Let us distinguish monotonic and non-monotonic
response functions.
\begin{itemize}
\item[-] When $\mu(\cdot)$ is monotonic on the interval $(0,S_{in})$, any
buffered configuration admits a unique positive equilibrium (function $\alpha_{\alpha,r}(\cdot)$ being
  decreasing, there exists an unique intersection of the graphs of
  $\mu(\cdot)/D$ and  $\alpha_{\alpha,r}(\cdot)$), that is
globally asymptotically stable. In terms of species survival and
stability, there is no difference with the single chemostat.

\item[-]  When $\mu(\cdot)$ is non-monotonic on the interval
  $(0,S_{in})$, one can consider values $\alpha>1$  such that
  $\lambda_{+}(\alpha D)<S_{in}$. Then, the washout
  equilibrium is attracting in the buffer vessel. For initial
  conditions in its attraction basin, the
  main tank behaves asymptotically as a single chemostat with a
  supply rate (or dilution rate) equal to $D/r$. For $r$ small enough one can have
  $S_{in}\notin \Lambda(D/r)$ or even $\Lambda(D/r)=\emptyset$. In those cases,
  the washout becomes an attracting equilibrium of the overall dynamics.

When the parameter $\alpha$ is such that the buffer has a unique
positive equilibrium, Theorem \ref{main_theo} shows that it is
  possible to have multiple equilibria. For instance, the set
  $\Gamma_{\alpha,r}(D)$ can have two connected components (as illustrated on
  Figure \ref{fig-phi2}). In this case, the system has three positive
  equilibria: the two endpoints of the first connected component and
  the left endpoint of the second one. According to Theorem
  \ref{main_theo}, the first
  and third equilibria are attracting while the second is not.
  Thus species persist in both tanks but the particular spatial
  structure can lead to several regimes of conversion at steady state,
  differently to a perfectly mixed vessel of the same total volume.
  Here, the buffer is playing the opposite role of a refuge: it
  highlights the fragility of a species to persist.
\end{itemize}

\bigskip

Finally, we have shown that the buffered configuration can have positive or negative
effects on the stability of an ecosystem, depending on the
characteristics of the buffer (size and flow rate). It can globally stabilize a
dynamics that is bi-stable in a perfectly mixed environment and avoid then
the washout of the biomass. At the opposite, a buffer can create a
multi-stability or even leads to a complete washout, while the
dynamics has a positive globally asymptotically stable equilibrium in perfectly mixed conditions.

\subsection{From a biotechnological point of view}
\label{subsection_biotechnological}
A typical field of biotechnological applications is the waste-water treatment with micro-organisms. 
For such industries, a usual objective is to reduce the output
concentration of substrate that is pumped out from the main tank. Typically, a
species that is selected to be efficient for low nutrient concentrations could
present a growth inhibition for large concentrations (its growth rate
being thus non-monotonic). Usually, the input
concentration $S_{in}$ is imposed by the industrial discharge and
cannot be changed, but the flow rate $Q$ can be manipulated. 
During the initial stage of continuous stirred bioreactors (that are
supposed to be perfectly mixed), the
biomass concentration is most often low (and the substrate concentration large).
This means that there exists a risk that the initial state
belongs to the attraction basin of the washout equilibrium if one
immediately applies the nominal flow rate $Q$. 
Such situation could also occurs during nominal functioning, under the
temporary presence of a toxic material that could rapidly deplete part of the
microbial population, and leave the substrate concentration higher than
expected. Those situations are well known from the practitioners: 
the process needs to be monitoring with the help of an automatic
control that makes the flow rate $Q$ decreasing in case of deviation
toward the washout. But such a solution requires an upstream storage
capacity when reducing the nominal flow rate, that could be
costly. Keeping a constant input flow rate is thus preferable.
An alternative is to
oversize the volume of the tank so that there is no longer
bi-stability and no need for a controller.
Compared to these two solutions, a design with a main tank and a
buffer (that guarantees a unique positive and globally asymptotically
stable equilibrium) presents several advantages:
\begin{itemize}
\item[-] it does not require to oversize the main tank,

\item[-] it does not require any upstream storage and the implementation of a controller,

\item[-] it allows to seed the initial biomass in the buffer tank only.

\end{itemize}

Notice that a by-pass of a single chemostat is also a way to reduce
the effective flow rate and to avoid a washout. It happens to 
be a
particular case of the buffered configuration with $V_{1}=0$.\\

Nevertheless, there is a price to pay to obtain the global stability
over the single bi-stable tank configuration:
\begin{itemize}
\item[i.] if the buffered configuration has the same total volume than the
  single chemostat, then the output concentration at steady state
  $S_{1}^{\star}$ would be higher than $\lambda_{-}(D)$, meaning that
  the buffered configuration would be less efficient than the single
  chemostat at its (locally asymptotically) stable positive
  equilibrium.

\item[ii.] to obtain the same nominal output $\lambda_{-}(D)$ with a
  buffered configuration, one needs to have a larger total volume.

\end{itemize}
However, considering a single chemostat of volume $V$ that presents a
bi-stability (that is when $\Lambda(D)\neq\emptyset$ and $\lambda_{+}(D)<S_{in}$), one can compare the minimal volume increment required to obtain a
single positive globally asymptotically stable equilibrium
by one of the following scenarios:
\begin{enumerate}
\item[] {\em Scenario 1:} enlarging the volume of the single chemostat by $\Delta V$.
\item[] {\em Scenario 2:} adding a buffer of volume $V_{2}$.
\end{enumerate}

For the first strategy, this amounts to have a new dilution rate equal
to $D/(1+\frac{\Delta V}{V})$. Then, the condition to be in Case 3 of Proposition
\ref{prop-chemostat} is to have
\[
S_{in} \in \Lambda\left(\frac{D}{1+\frac{\Delta V}{V}}\right) \ ,
\]
or equivalently
\begin{equation}
\label{formule_deltaV}
\frac{\Delta V}{V} > \left(\frac{\Delta V}{V}\right)_{\inf}=\frac{D}{\mu(S_{in})}-1 \ .
\end{equation}

For the second strategy, one has to choose first the dilution rate
$D_{2}=Q_{2}/V_{2}$ of the buffer (with $Q_{2}<Q$).
For any positive number $D_{2} < \mu(S_{in})$, there exists a unique positive
equilibrium $(S_{2}^{\star}(D_{2}), S_{in}-S_{2}^{\star}(D_{2}))$ in
the buffer, where
\[
S_{2}^{\star}(D_{2})=\lambda_{-}(D_{2}) \; < \; \bar
s=\lambda_{-}(\mu(S_{in})) \ .
\]
The Proposition \ref{prop5}, given in the Appendix, provides an explicit lower bound on the volume $V_{2}$ to ensure a unique globally
exponentially stable positive equilibrium from any initial condition
with $S_{2}(0)>0$. Furthermore, this bound is
necessarily such that
\[
\left(\frac{V_{2}}{V}\right)_{\inf} < \left(\frac{\Delta
    V}{V}\right)_{\inf} \ .
\]
The benefit of Scenario 2 over Scenario 1 in terms of volume increment
will be numerically demonstrated in Section \ref{section_example}.

\section{ A numerical illustration}
\label{section_example}

In this section, we illustrate numerically the stabilizing effect of a
buffer.
We consider the case of the single chemostat model that
presents a bi-stability (see the discussion in \ref{subsection_attracting_washout}), with a
non-monotonic uptake function given by the Haldane expression
(\ref{Haldane}).
One can easily check that for this function the set
$\Lambda(D)$ defined in (\ref{Lambda}) is non-empty exactly when the condition
\[
\bar\mu/D > 1 + 2\sqrt{\frac{K}{K_{I}}}
\]
is fulfilled. Then, $\lambda_{-}(D)$, $\lambda_{+}(D)$ are given by the
following expressions:
\[
\lambda_{\pm}(D) = \frac{K_{I}(\bar\mu/D-1)\pm\sqrt{K_{I}^{2}(\bar\mu/D-1)^{2}-4KK_{I}}}{2} \ .
\]
Bi-stability occurs when the condition $S_{in}>\lambda_{+}(D)$
is fulfilled (case 2 of Proposition \ref{prop-chemostat}).\\

Recall from Section \ref{section3}, that for the Haldane function, the
solutions of  the equation (\ref{geometric_condition}) are
roots of a polynomial of order $3$ with at most three solutions of
(\ref{polynom}). There is a most one double root, which implies that the set
${\cal S}_{r,\alpha}(D)$ possesses at most one element.
Proposition \ref{prop1} (case II), given in the Appendix, helps to
characterize the set $\overline R_{\alpha}(D)$ depending on the
subsets $R^{-}_{\alpha}(D)$, $R^{+}_{\alpha}(D)$ that are defined in
this Proposition:
\begin{itemize}
\item[-] the set $R^{+}_{\alpha}(D)$ is a singleton, because there are at
  most three equilibria,
\item[-]  $R^{-}_{\alpha}(D)\cap R^{+}_{\alpha}(D)=\emptyset$ because
  ${\cal S}_{r,\alpha}(D)$ possesses at most one element,
\item[-] when $R^{-}_{\alpha}(D)$ is non-empty, one has $\max
  R^{-}_{\alpha}(D)<\min R^{+}_{\alpha}(D)$: for any $r \in (\min R^{-}_{\alpha}(D),\max R^{-}_{\alpha}(D))$, 
equation (\ref{geometric_condition}) has at least three solutions on an
interval $I$, and for $r
\in (\min R^{+}_{\alpha}(D),1)$ at least two on another interval $J$, where $I$
and $J$ are disjoint. If $\max R^{-}_{\alpha}(D)\geq\min
R^{+}_{\alpha}(D)$, there would
exist at least five solutions of equation (\ref{geometric_condition}) on $(0,S_{in})$.
\end{itemize}

We study now the set of ``stable'' buffered configurations ${\cal
  C}_{D}$ as the set of pairs $(\alpha,r)$ such that the
buffered chemostat model admits a unique positive equilibrium. 
The upper boundary of ${\cal  C}_{D}$ is thus given by the curve
\[
\alpha \in (0,\mu(S_{in})/D] \; \mapsto \; \bar r_{D}(\alpha)
\]
where $\bar r_{D}(\alpha)$ is the single element of the set
$R^{+}_{\alpha}(D)$.
Notice that the limiting case $\alpha D=\mu(S_{in})$ can
have also global stability (see Lemma \ref{lemma} in the Appendix).
The number $\bar r_{D}(\alpha)$ can then be determined numerically as the
unique minimizer of the function
\[
F_{\alpha}(r,s)=\left(\mu(s)/D-\phi_{\alpha,r}(s)\right)^{2}+
\left(\mu^{\prime}(s)/D-\phi_{\alpha,r}^{\prime}(s)\right)^{2}
\]
on $(0,1)\times \{s\in (\lambda^{-}(D),S_{in}) \mbox{ s.t. }
(s-\lambda^{+}(D))(\lambda^{+}(D)-\ul S(\alpha))\geq 0\}$
that is, for the Haldane function:
\[
F_{\alpha}(r,s)=\left(\frac{(\bar\mu/D)
    s}{K+s+s^{2}/K_{I}}-\frac{1}{r}+\alpha\frac{1-r}{r}\frac{S_{in}-\lambda_{-}(\alpha
    D)}{S_{in}-s}\right)^{2}\!\!+
\left(\frac{\bar\mu/D(K-s^{2}/K_{I})}{(K+s+s^{2}/K_{I})^{2}}+\alpha\frac{1-r}{r}\frac{S_{in}-\lambda_{-}(\alpha
    D)}{(S_{in}-s)^{2}}\right)^{2}
\]
where $\ul S(\alpha)$ is defined in (\ref{ulS}).
For the parameters given in Table \ref{table1}, we have computed
numerically the domains ${\cal C}_{D}$ for different values of $S_{in}$,
depicted on Figure \ref{figdomain}.
\begin{table}[h!]
\begin{center}
\begin{tabular}{|c|c|c|c||c|c|}
\hline
$\bar\mu$ & $D$ & $K$ & $K_{I}$\ & $\lambda_{-}(D)$ & $\lambda_{+}(D)$\\
\hline
$12$ & $1$ & $1$ & $0.8$ & $\simeq 0.103$ & $\simeq 0.777$\\
\hline
\end{tabular}
\end{center}
\caption{Parameters of the Haldane function and the corresponding
  values of $\lambda_{-}(D)$ , $\lambda_{+}(D)$.\label{table1}}
\end{table}
\begin{figure}[h!]
\begin{center}
\includegraphics[height=6cm]{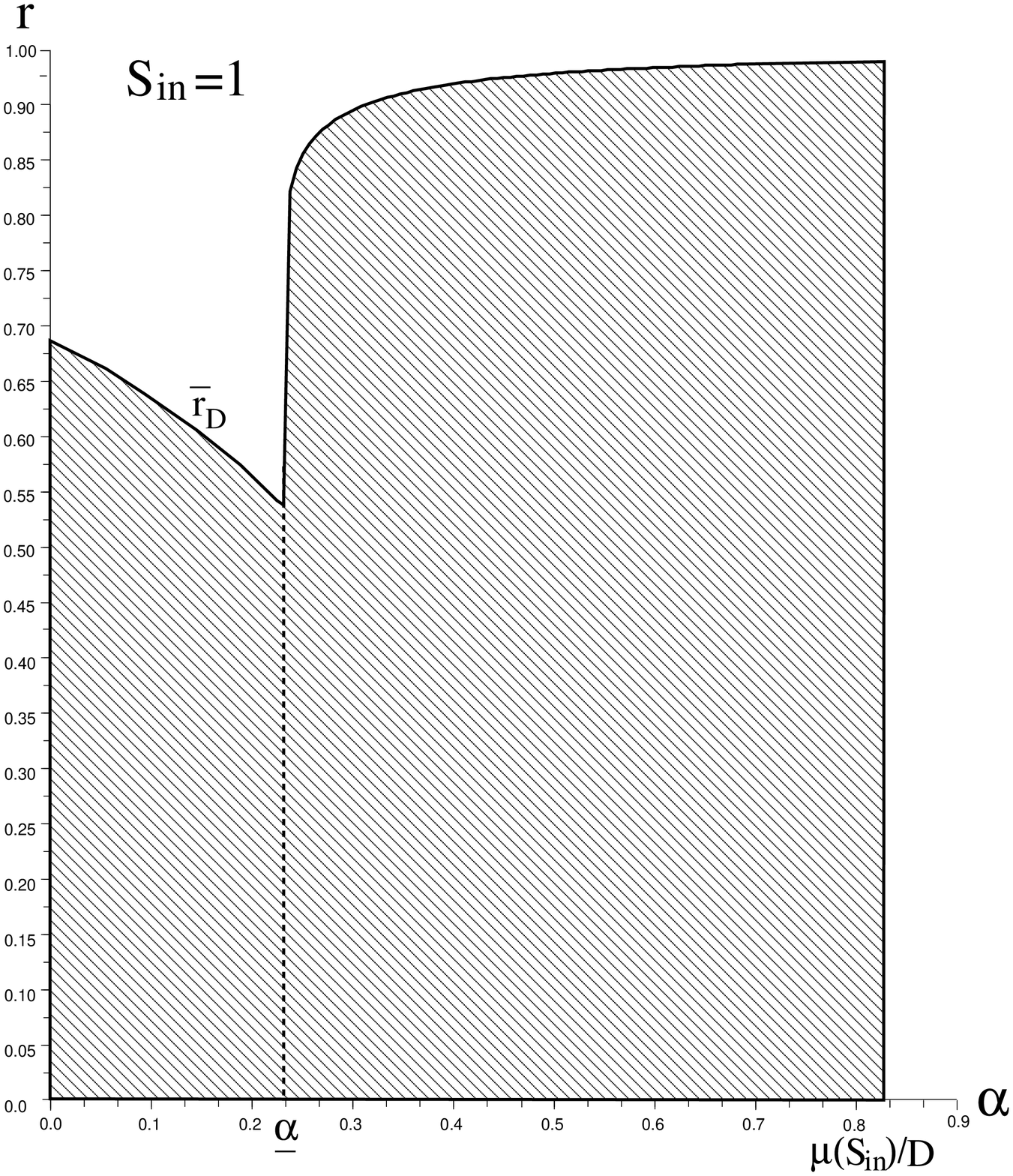}
\hspace{2mm}
\includegraphics[height=6cm]{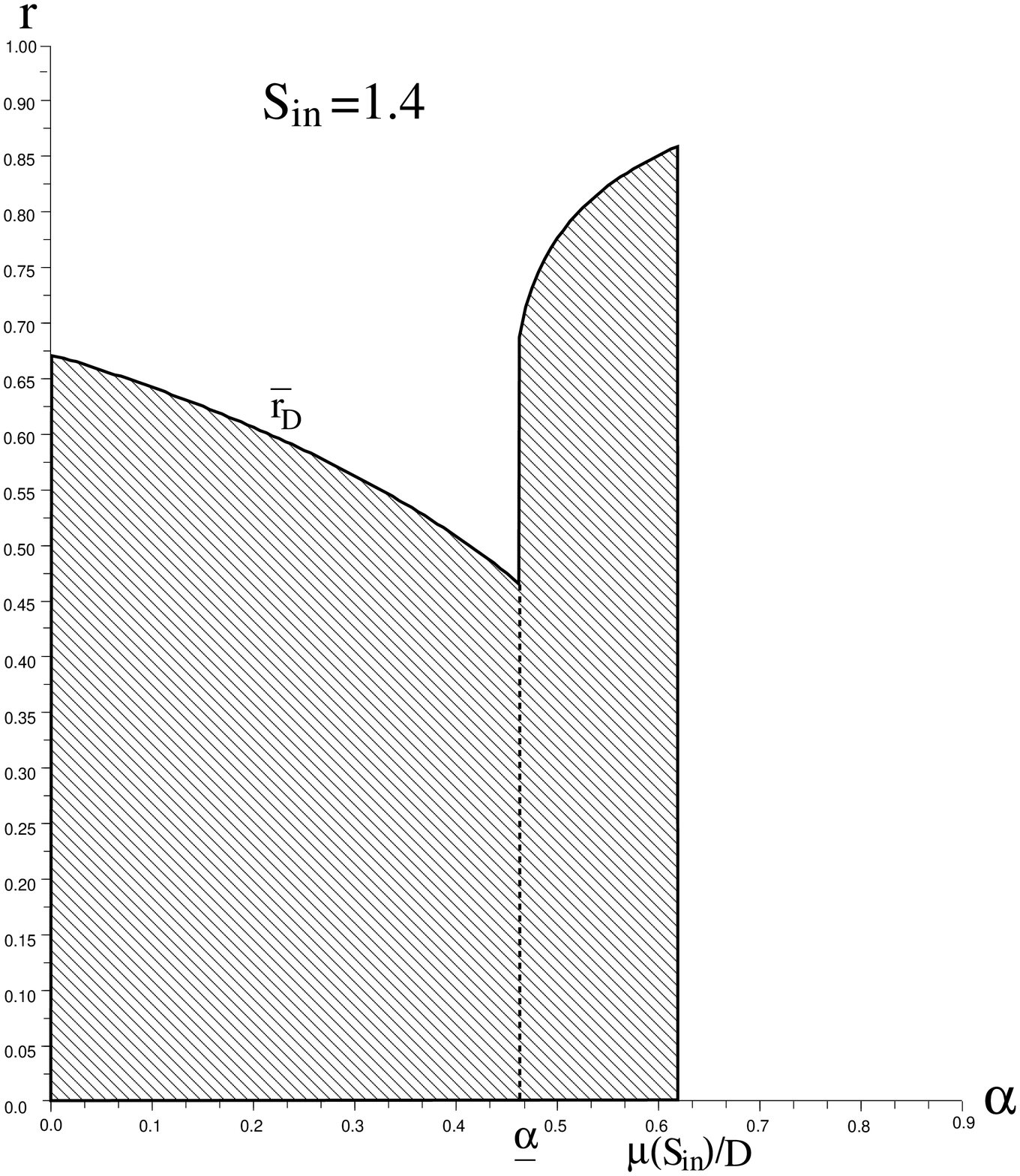}
\hspace{2mm}
\includegraphics[height=6cm]{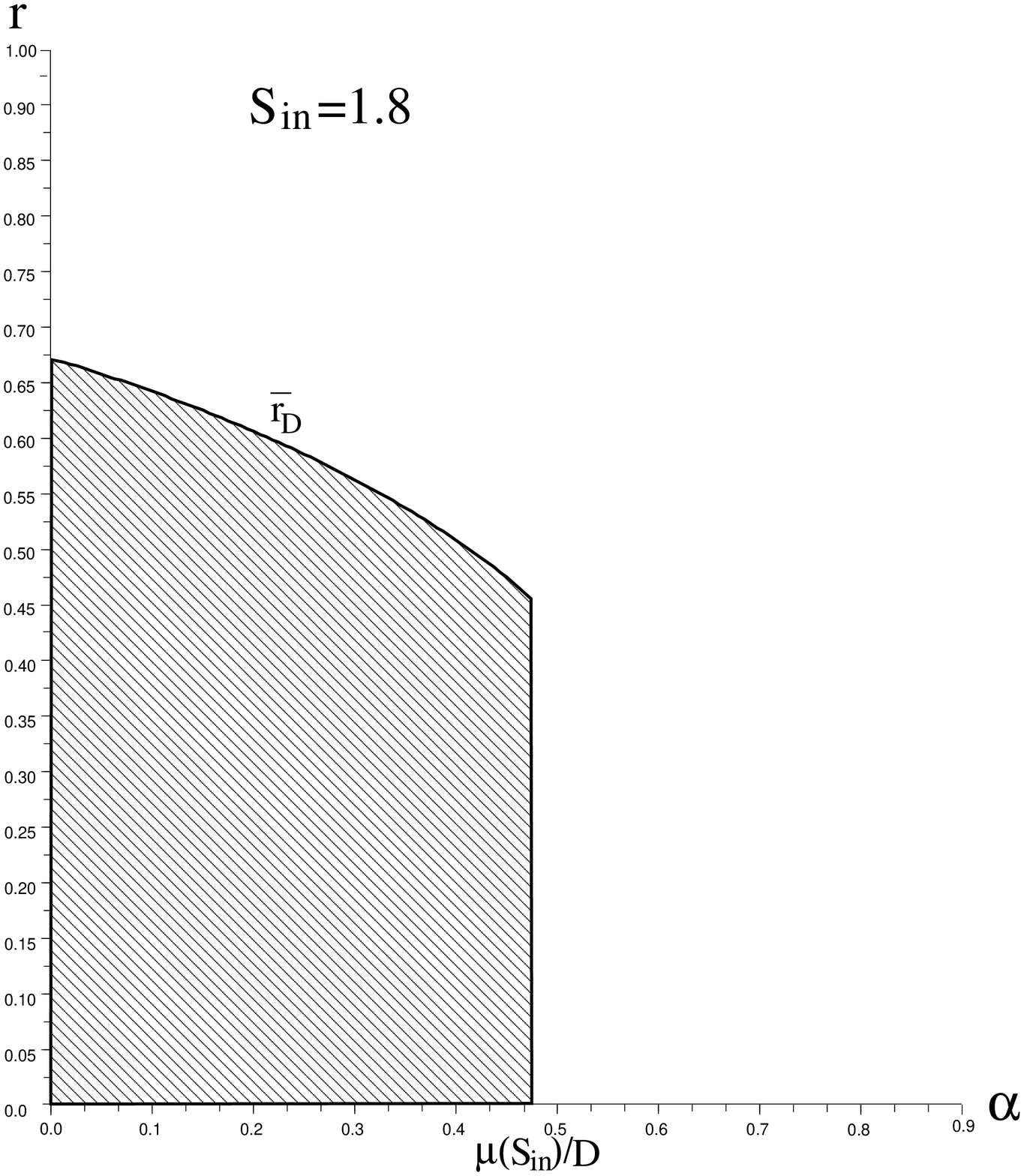}
\caption{Domain ${\cal C}_{D}$ of stable configurations for different
  values of $S_{in}$.\label{figdomain}}
\end{center}
\end{figure}
One can see that the map $\alpha
\mapsto \bar r_{D}(\alpha)$ is discontinuous at $\alpha=\ul \alpha$, where 
$\ul \alpha$ is such that $\ul S(\ul \alpha)=\lambda_{+}(D)$ (when it
exists), as mentioned in  Remark \ref{remark1}.
On Figure \ref{figlimiting} one can see that the two limiting
hyperbolas $H_{\alpha,\bar r(\alpha)}$ about $\ul \alpha$
are different in a such a case.
\begin{figure}[ht]
\begin{center}
\includegraphics[width=6cm]{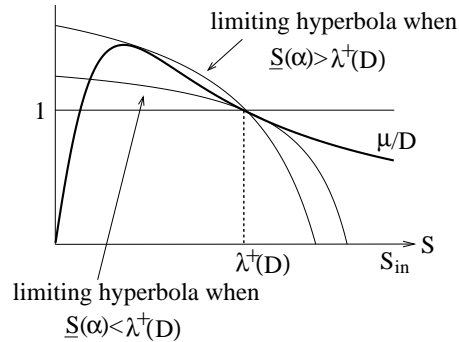}
\caption{The limiting hyperbolas $H_{\alpha,\bar
  r(\alpha)}$ about $\alpha=\ul \alpha$ (for $S_{in}=1.4$).\label{figlimiting}}
\end{center}
\end{figure}
So, this study reveals the role of the input concentration $S_{in}$ on the
shape of the domain ${\cal C}_{D}$.\\

Finally, we have compared the two scenarios discussed in
\ref{subsection_biotechnological} for improving the stability of the single
chemostat, by enlarging its volume or adding a buffer,
given respectively by formulas (\ref{formule_deltaV}) and
(\ref{condV2}).
For the parameters given in Table \ref{table1}, the numerical
comparison is reported on Figure \ref{fig_deltaV} as a function of the
input concentration $S_{in}$.
\begin{figure}[h!]
\begin{center}
\includegraphics[width=6cm]{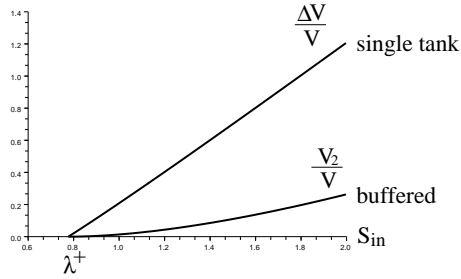}
\caption{Comparison of the minimal increase of volume equired to obtain the global
  stability (as function of the input concentration).\label{fig_deltaV}}
\end{center}
\end{figure}\\
As expected, the buffered chemostat requires less volume
augmentation, but one can also discover that 
this advantage becomes more significant as the input
concentration $S_{in}$ is higher.
Finally, this study demonstrates on a concrete example the flexibility of the buffered
chemostat in the choice of possible configurations, with two
parameters to be tuned (instead of one for the single chemostat).

\section{Conclusion}

The present analysis illustrates how the addition of a buffer to
chemostat alters the multiplicity and stability of their equilibria.
This property has several impacts on theoretical
ecology as well as for bio-industrial applications.
\begin{itemize}
\item[-] From an ecological viewpoint, a spatial pattern with a
buffer can explain why a species can persist in an environment that is
unfavorable if it was perfectly mixed. On the opposite, the
emergence of a buffer with particular characteristics can destabilize 
a regime that is stable under perfectly mixed conditions, and could lead to
the extinction of the species. Nevertheless, such a case occurs only for
``fragile'' species with non-monotonic response function.
\item[-] For industrial applications, such as waste-water purification
  or pharmaceutic production, a buffered configuration of two tanks,
  instead of one or serial or parallel interconnections, present
  several advantages when there is an inhibition in the growth
  rate. It provides an easy and robust way to prevent the washout of the
  biomass in the process, without requiring upstream storage or
  real-time controller. 
\end{itemize}
The numerical study has also revealed other interesting characteristics of
the buffered chemostat.
First, the size of the ``buffer'' or the additional tank that provide such properties could be relatively small.
Secondly, the shape of the set of buffered configurations that provide a
unique (globally asymptotically stable) positive equilibrium depends
on the density of the supplied resource, with a threshold that makes
this shape non smooth. \\

Finally, those results provide new insights on the role of spatial structures
in resource/consumer models for natural ecosystems, and new
potential strategies for the design of industrial bioprocesses. 
Of course, more complex interconnections could be considered, with for
instance an additional output from the buffer. However, the main contribution of the
present work is to show that a simple configuration with only two
parameters can change radically the overall dynamic behavior. The
buffered chemostat appears to be the simplest pattern that can provide global
stability, while any serial or parallel configurations cannot do.

Our study considered a single strain.
According to the Competitive Exclusion Principle, it is not
(generically) possible to have more than one species persisting in the buffer tank, but this does not prevent to have
coexistence with another species in the main tank,
which is not possible with a single chemostat.
Consequently, it might be relevant to study the dynamics of the buffered
chemostat with different persistent species in the buffer and in the
main tank.\\

\section*{Acknowledgments}
The authors are grateful to INRA and INRIA supports
within the French VITELBIO (VIRtual TELluric BIOreactors) research
program.
The authors thank also Prof.~Denis Dochain, CESAME,
Univ. Louvain-la-Neuve, for fruitful discussions.
The authors would also like to thank the anonymous referees for their
relevant suggestions for improvements of our initial work.\\

\section*{Appendix}

\bigskip

{\bf Proof of Lemma \ref{LemmaSerialParallel}.}

In the serial connection, the dynamics of the first tank of
  volume $V_{1}$ is given by equations (\ref{chemostat0}) where $V$ is
  replaced by $V_{1}\leq V$.
  Its dilution rate is then equal to $Q/V_{1}$, that is greater than
  $Q/V$ and consequently one has $S_{in}\notin
  \Lambda(Q/V_{1})$. According to Proposition \ref{prop-chemostat},
  only Cases 1 or 2 can occur in the first tank.

In the parallel connection, the dynamics of each tank of volume
  $V_{i}$ and flow rate $Q_{i}$ is given by equations
  (\ref{chemostat0}) where $V$  and $Q$ are replaced by $V_{i}$ and $Q_{i}$.
  Denote $r_{i}=V_{i}/V$ and $\alpha_{i}=Q_{i}/Q$, and notice that one has
  $\sum_{i}r_{i}=\sum_{i}\alpha_{i}=1$. Then, the dilution rate
  $D_{i}$ in the tank  $i$ is equal to $\alpha_{i}/r_{i}D$. 
  According to Proposition
  \ref{prop-chemostat}, a necessary condition for having the washout
  equilibrium repulsive in each tank is to have $D_{i}<D$
  for any $i$, that is $\alpha_{i}<r_{i}$, which contradicts $\sum_{i}r_{i}=\sum_{i}\alpha_{i}=1$.\qed\\

\bigskip

Before giving the proof of Theorem \ref{main_theo}, we present in the
next proposition a
series of results concerning the multiplicity of equilibria and the
characterization of the sets $\overline{R}_{\alpha}(D)$ defined
in (\ref{defRbar}).

\begin{proposition}
\label{prop1}
Assume that Hypotheses A1 are fulfilled. Fix $D>0$ and
take a positive number $\alpha$ such that $\Lambda(\alpha D)\neq
\emptyset$ and $\lambda_{-}(\alpha D) < S_{in}$.
Let $S_{2}^{\star}(\alpha) \in (0,S_{in})$ be such
that $\mu(S_{2}^{\star}(\alpha))=\alpha D$. Then, for any $r \in
(0,1)$ there exists an equilibrium $(S_{1}^{\star},S_{in}-S_{1}^{\star},S_{2}^{\star}(\alpha),S_{in}-S_{2}^{\star}(\alpha))$
of (\ref{chemostat2b}), with
\begin{equation}
\label{S1star}
S_{1}^{\star} \in \left\vert\begin{array}{ll}
(\underline S(\alpha),S_{in}) & \mbox{when } \Lambda(D)=\emptyset
\mbox{ or } \underline S(\alpha)\notin \Lambda(D) \ ,\\
{[\lambda_{-}(D),\underline S(\alpha)]} & \mbox{when }
\underline S(\alpha))\in \Lambda(D) \ .
\end{array}\right.
\end{equation}
Furthermore, the set $R_{\alpha}(D)$ defined in (\ref{defRalpha}) is not reduced to a singleton
when it is non-empty. We distinguish two different cases:

\begin{itemize}
\item[] {\em Case I:} $\Lambda(D)=\emptyset$ or $\lambda_{-}(D)\geq
  S_{in}$ or $\lambda_{+}(D)\geq S_{in}$. One has
\[
\overline{R}_{\alpha}(D) = \left|
\begin{array}{ll}
(0,1) & \mbox{when } R_{\alpha}(D)=\emptyset ,\\
(0,1)\setminus\left[\min R_{\alpha}(D),\max
  R_{\alpha}(D)\right]  & \mbox{when } R_{\alpha}(D)\neq \emptyset .
\end{array}\right.
\]
For $r\notin \overline{R}_{\alpha}(D)$, the exist at least three equilibria
  with $S_{1}^{\star} \in (\underline S(\alpha),S_{in})$ when
  $\Lambda(D)=\emptyset$ or $S_{1}^{\star} \in
  (\lambda_{-}(D),\underline S(\alpha))$ when
  $\Lambda(D)\neq\emptyset$.\\

\item[] {\em Case II:} $\lambda_{+}(D)< S_{in}$.
We consider the partition of the set $R_{\alpha}(D)$:
\begin{eqnarray}
\label{defR-}
R^{-}_{\alpha}(D) & =  & \{ r \in (0,1) \; \vert \; \exists s \in
{\cal S}_{\alpha,r}(D) \mbox{ with } (s-\ul S(\alpha))(\lambda_{+}(D)-\ul
S(\alpha))< 0 \} \ ,\\[2mm]
\label{defR+}
R^{+}_{\alpha}(D) & = & \{ r \in (0,1) \; \vert \; \exists s \in
{\cal S}_{\alpha,r}(D) \mbox{ with } (s-\lambda_{+}(D))(\lambda_{+}(D)-\ul
S(\alpha))\geq 0 \} \ .
\end{eqnarray}
Then, the set $R^{+}(\alpha)$ is non-empty, and the set $R^{-}(\alpha)$ is not
reduced to a singleton when it is non-empty. One has
\[
\overline{R}_{\alpha}(D)=\left|\begin{array}{ll}
(0,\min R^{+}(\alpha)) & \mbox{when } R^{-}(\alpha)=\emptyset \ ,\\
(0,\min R^{+}(\alpha)) \; \cap \; 
(0,1)\setminus [\min R^{-}(\alpha),\max R^{-}(\alpha)] & \mbox{when }
R^{-}(\alpha)\neq \emptyset \ .
\end{array}\right.
\]
For any $r\in (\min R^{+}(\alpha),1)$, there exist at least
two equilibria such that $(\ul
S(\alpha)-S^{\star}_{1})(\lambda_{+}(D)-\ul S(\alpha))\geq 0$, and at
least four for $r$ in a subset of $(\min R^{+}(\alpha),1)$
when $R^{+}(\alpha)$ is not reduced to a singleton.

When $R^{-}(\alpha)$ is non-empty, for any
$r \in (\min R^{-}(\alpha),\max R^{-}(\alpha))$, there exist at least
three equilibria 
such that $(\ul
S(\alpha)-S^{\star}_{1})(\lambda_{+}(D)-\ul S(\alpha))<0$.
\end{itemize}
\end{proposition}

\bigskip
\noindent {\em Remark.}
In Case II, the tangency of the graphs of $\phi_{\alpha,r}$ and $\mu$ occurs for a certain $r$ with an abscissa that is located
\begin{itemize}
\item[-] either at the right of $\lambda_{+}$ when $\ul
  S(\alpha)<\lambda_{+}(D)$,
\item[-] either at the left of $\lambda_{+}$ when $\ul
  S(\alpha)>\lambda_{+}(D)$.
\end{itemize}
These cases correspond to the subset $R_{\alpha}^{+}(D)$ while the subset
$R_{\alpha}^{-}(D)$ corresponds to other tangencies that could occur (but
that do not necessarily exist) on either side of $\ul S(\alpha)$.\\

\noindent {\bf Proof of Proposition \ref{prop1}.}

Fix $D$ and $\alpha$ such that $\Lambda(\alpha D)\neq\emptyset$
and $\lambda_{-}(\alpha D)<S_{in}$. For simplicity, we denote by $S_{2}^{\star}$ and
$\ul S$ the values of $S_{2}^{\star}(\alpha)$ and  $\ul S(\alpha)$,
with $S_{2}^{\star}$ such that $\mu(S_{2}^{\star})=\alpha D$.
For each $r \in(0,1)$, we define the function
\[
f_{r}(s)=D\phi_{\alpha,r}(s)-\mu(s) \ .
\]
A non-negative equilibrium for the first tank has then to satisfy $f_{r}(S_{1}^{\star})=0$.\\

One can easily check that $\phi_{\alpha,r}(\ul S)=1$ whatever the
value of $r
\in (0,1)$. The function $\phi_{\alpha,r}(\cdot)$ being decreasing, one has
$\phi_{\alpha,r}(s)>1$ for $s<\ul S$ and $\phi_{\alpha,r}(s)<1$ for $s>\ul S$.
For convenience, we shall also consider the function
\begin{equation}
\label{gamma}
\gamma(s)=\frac{\ul S -s}{\ul S - S_{in}+(S_{in}-s)\mu(s)/D}
\end{equation} 
that is defined on the set of $s \in (0,S_{in})$ such that
$(S_{in}-s)\mu(s)\neq S_{in}-\ul S$. On this set, 
one can easily check that the following equivalence is fulfilled
\[
f_{r}(s)=0 \Longleftrightarrow \gamma(s)=r \ .
\]
From (\ref{gamma}), one can also write
\[
\gamma(s)=\frac{(\phi_{\alpha,r}(s)-1)\frac{r}{1-r}}{(\phi_{\alpha,r}(s)-1)\frac{r}{1-r}-1+\mu(s)/D}
\]
and deduce the property
\begin{equation}
\label{gammaprime}
\gamma^{\prime}(s)=0 \Longleftrightarrow \phi_{\alpha,r}^{\prime}(s)(\mu(s)/D-1)=(\phi_{\alpha,r}(s)-1)\mu^{\prime}(s)/D \ .
\end{equation}
Recursively, one obtains for every integer $n$
\[
\left\{ \frac{d^{p}\gamma}{ds^{p}}(s)=0 \, , \, p=1\cdots n\right\}
\Longleftrightarrow 
\left\{ D\frac{d^{p}\phi_{\alpha,r}}{ds^{p}}(s)(\mu(s)-D)=
(D\phi_{\alpha,r}(s)-D)\frac{d^{p}\mu}{ds^{p}}(s)  \, , \, p=1\cdots
n\right\} \ .
\]
Consequently, the set ${\cal S}_{\alpha,r}$ defined in (\ref{calS}) 
can be characterized as
\[
{\cal S}_{\alpha,r} = 
\left\{ s \in (\lambda_{-},S_{in}) \mbox{ s.t. } \gamma(s)=r \mbox{
  and }
\min\left\{ n \in \Nset^{\star} \, \vert \, \frac{d^{n}\gamma}{ds^{n}}(s)\neq 
0 \right\} \mbox{ is even} \right\}
\]
or equivalently
\begin{equation}
\label{newcalS}
{\cal S}_{\alpha,r} = 
\left\{ s \in (0,S_{in}) \mbox{ s.t. } \gamma(s)=r \mbox{ is
  a local extremum  } \right\} \ .
\end{equation}

We distinguish several cases depending on the position of
$\ul S$ with respect to the set $\Lambda(D)$. In the following, we simply denote
$\Lambda$, $\lambda_{\pm}$ and $R_{\alpha}$ for $\Lambda(D)$,
$\lambda_{\pm}(D)$ and $R_{\alpha}$(D) respectively.\\

{\em Case I.} 
\medskip\\
When $\Lambda=\emptyset$ or $\lambda_{-}\geq S_{in}$, the function
$f_{r}$ is strictly positive on the interval $[0,\ul S]$.
On the interval $J=(\ul S,S_{in})$, the function $\gamma(\cdot)$ is
well defined with $\gamma(J)=(0,1)$, $\gamma(\ul S)=0$ and
$\gamma(S_{in})=1$. 
Consequently, there exists at least one solution of
$f_{r}(s)=0$, that necessarily belongs to the interval $J$.
If $R_{\alpha}=\emptyset$, $\gamma(\cdot)$ is increasing and there
exists a unique solution of $\gamma(S_{1}^{\star})=r$ whatever
is $r$. 
Notice that when the function $\mu(\cdot)$ is increasing on
$[0,S_{in}]$ (which is necessarily the case when $\lambda_{-}\geq
S_{in}$), one has necessarily $R_{\alpha}=\emptyset$, because the
function $f_{r}$ is decreasing.
Otherwise, property (\ref{newcalS}) implies that
$\gamma$ admits local extrema, and $\min R_{\alpha}$ and $\max R_{\alpha}$ are
respectively the smallest local minimum and largest local maximum of
the function $\gamma$ on the interval $J$. Consequently, the set
$R_{\alpha}$ cannot be reduced to a singleton.
Then, uniqueness of $S_{1}^{\star}$ is achieved exactly for $r$ that
does not belong to $[\min R_{\alpha},\max R_{\alpha}]$. For
any $r \in (\min R_{\alpha},\max R_{\alpha})$, there are at least three
solutions, that all belong to $J$,  by the Mean Value Theorem.
\medskip \\
When $\lambda_{-}<S_{in}\leq \lambda_{+}$, we distinguish two
sub-cases:
\begin{itemize}
\item[] $\ul S\leq \lambda_{-}$: the function $f_{r}$ is strictly positive
  on $[0,\ul S)$ and strictly negative on
  $(\lambda_{-},S_{in})$. Furthermore, $f_{r}$ is decreasing on
  $[\ul S,\lambda_{-}]$. So there exists a unique root $S_{1}^{\star}$
  of $f_{r}$ that necessarily belongs to $[\ul S,\lambda_{-}]$ (and
  the set $R_{\alpha}$ is empty).
\item[] $\ul S > \lambda_{-}$: the function $f_{r}$ is strictly positive
  on $[0,\lambda_{-}]$ and strictly negative on
  $[\ul S,S_{in})$.
On the interval $I=(\lambda_{-},\ul S)$, the function
$\gamma(\cdot)$ is well defined and $\gamma(I)=(0,1)$ with $\gamma(\lambda_{-})=1$ and $\gamma(\ul
S)=0$. If $R_{\alpha}$ is empty, then $\gamma(\cdot)$ is
decreasing on $I$, and for any
$r\in(0,1)$ there exits a unique $S_{1}^{\star}$ such that
$\gamma(S_{1}^{\star})=r$. 
If $R_{\alpha}$ is non-empty, property (\ref{newcalS}) implies that
$\gamma$ admits local extrema, and $\min R_{\alpha}$ and $\max R_{\alpha}$ are
respectively the smallest local minimum and largest local maximum of
the function $\gamma$ on the interval $I$.
Then, uniqueness of $S_{1}^{\star}$ on $J$ is achieved exactly for $r$ that
does not belong to $[\min R_{\alpha},\max R_{\alpha}]$. For
any $r \in (\min R_{\alpha},\max R_{\alpha})$, there are at least three
solutions, that all belong to the interval $I$, by the Mean Value
Theorem.
\end{itemize}

\bigskip

{\em Case II.}
\medskip\\
Notice that in this case ($\lambda_{+}<S_{in}$) the function $\mu$ is
non-monotonic. We consider three sub-cases depending on the relative
position of $\ul S$ with respect to $\lambda_{+}$.
\medskip\\
{\em Sub-case 1:} $\ul S < \lambda_{+}$. As for Case I, we
distinguish:
\begin{itemize}
\item[] $\ul S \leq \lambda_{-}$: one has
$f_{r}(\ul S)\geq 0$ and $f_{r}(S)<0$ for any $S \in
\Lambda$. $f_{r}(\cdot)$  being decreasing on $[0,\lambda_{-}]$, one
deduces that there exists exactly one solution $S_{1}^{\star}$ of 
$f_{r}(S)=0$ on the interval $[0,\lambda_{+}]$, whatever is $r$.
Furthermore, this solution has to belong to $[\ul S,\lambda_{-}]$.
The functions $\phi_{r}(\cdot)$ and $\mu(\cdot)$ being respectively
decreasing and increasing on this interval, one has necessarily
$\gamma^{\prime}(S_{1}^{\star})\neq 0$ and then $R^{-}_{\alpha}=\emptyset$.
\item[] $\ul S > \lambda_{-}$: one has $f_{r}(S)>0$
for any $S \in [0,\lambda_{-}]$, and  $f_{r}(S)<0$ for
any  $S \in [\ul S,\lambda_{+}]$.
On the interval $I=(\lambda_{-},\ul S)$, the function
$\gamma(\cdot)$ is well defined and $\gamma(I)=(0,1)$ with $\gamma(\lambda_{-})=1$ and $\gamma(\ul
S)=0$. If $R^{-}_{\alpha}$ is empty, then $\gamma(\cdot)$ is
decreasing on $I$, and for any
$r\in(0,1)$ there exits a unique $S_{1}^{\star} \in I$ such that
$\gamma(S_{1}^{\star})=r$. 
If $R^{-}_{\alpha}$ is non-empty, property (\ref{newcalS}) implies that
$\gamma$ admits local extrema.
Similarly to Case I, we obtain by the Mean Value Theorem that there
exists exactly one solution $S_{1}^{\star}$ of $\gamma(s)=r$ on the interval $[0,\lambda_{+}]$
for any $r\notin [\min R^{-}_{\alpha},\min R^{-}_{\alpha}]$, and there
are at least three solutions for $r \in (\min R^{-}_{\alpha},\min
R^{-}_{\alpha})$.
\end{itemize}
Differently to Case I, we have also to consider the interval $K=(\lambda_{+},S_{in})$ where the function
$\gamma(\cdot)$ is well defined and positive with
$\gamma(\lambda_{+})=1$ 
and $\lim_{s\to S_{in}}\gamma(s)=1$. We define
\[
r^{+}=\min\{ \gamma(s) \, \vert \, s \in K\}
\]
that belongs to $(0,1)$. Then $r^{+}$ belongs to $R^{+}_{\alpha}$, and
for any $r<r^{+}$ there is no solution of
$\gamma(s)=r$ on $K$. Thus $r^{+}$ is the minimal element of $R^{+}_{\alpha}$.
By the Mean Value Theorem there are at least
two solutions of $\gamma(s)=r$ on $K$ when $r>r^{+}$.
When $R^{+}_{\alpha}$ is not reduced to a singleton, the function
$\gamma$ has at least on local maximum $r_{M}$ and one local minimum
$r_{m}$, in addition to $r^{+}$. By the Mean Value Theorem, there are
at least four solutions of $\gamma(s)=r$ on $K$ for $r \in
(r_{m},r_{M})$.
\medskip\\
Finally, we have shown that the set $R^{+}_{\alpha}$ is non-empty, and
that the uniqueness of the solution
of $\gamma(S_{1}^{\star})=r$ occurs exactly for values of $r$ that do not
belong to the set $[\min R^{-}_{\alpha},\max R^{-}_{\alpha}] \cup [\min
R^{+}_{\alpha},1]$.
\bigskip\\
{\em Sub-case 2:} $\ul S = \lambda_{+}$.
One has $f_{r}(\ul S)=0$ for any $r$, so there exists a positive equilibrium with
$S_{1}^{\star}=\ul S$. $f_{r}(S)>0$ for any $S \in [0,\lambda_{-}]$ 
and the function $\gamma(\cdot)$ is well defined
on $I\cup J=(\lambda_{-},\ul S)\cup (\ul S,S_{in})$
with $\gamma(I\cup J)=(0,1)$, $\gamma(\lambda_{-})=1$ and $\lim_{s\to
  S_{in}}\gamma(s)=1$. Using the L'H\^{o}pital's rule, we show that the
function $\gamma(\cdot)$ can be continuously extended at $\ul S$:
\[
\lim_{s\to \ul S} \gamma(s)=
\lim_{s\to \ul S}
\frac{-1}{-\mu(s)/D+(S_{in}-s)\mu^{\prime}(s)/D}=
\frac{1}{1-(S_{in}-\ul S)\mu^{\prime}(\ul S)/D} \ .
\]
Note that $\mu^{\prime}(\ul S)<0$ so that $\gamma(\ul
S)$ belongs to $(0,1)$, and we pose 
\[
\bar r = \min\{ \gamma(s) \, \vert \, s \in (\lambda_{-},S_{in})\} \ .
\]
Then, for $r<\bar r$, there is no solution of $\gamma(s)=r$ on
$(\lambda_{-},S_{in})$, and $\ul S$ is the only solution of
$f_{r}(s)=0$ on $(0,S_{in})$. On the contrary, for $r>\bar r$, 
there are at least two solutions of $\gamma(s)=r$ on
$(\lambda_{-},S_{in})$
and the dynamics has at least two positive equilibria.\\

Similarly, the function $\gamma(\cdot)$ is $C^{1}$ on 
$(\lambda_{-},S_{in})$ because it is differentiable at
$\ul S$:
\[
\gamma^{\prime}(\ul S)=D\frac{(S_{in}-\ul
  S)\mu^{\prime\prime}(\ul S)-2\mu^{\prime}(\ul S)}{
[D-(S_{in}-\ul S)\mu^{\prime}(\ul S)]^{2}} 
\]
(and recursively as many time differentiable as the function
$\mu(\cdot)$ is, minus one).
Then $\bar r$ is the minimal element of the set $R^{+}_{\alpha}$, and
the set $R^{-}_{\alpha}$ is empty by definition.
As previously, when $R^{+}_{\alpha}$ is not reduced to a singleton, 
$\gamma(s)=r$ has at least four solutions for $r$ in a subset of
$(\min R^{+}_{\alpha},1)$.
\bigskip\\
Sub-case 3: $\ul S > \lambda_{+}$. We proceed similarly as in sub-case 1.
Note first that there is no solution of $f_{r}(s)=0$ on the intervals
$(0,\lambda_{-})$ and
$(\lambda_{+},\ul S)$ whatever is $r$.\\

On the set $\Lambda$, $\gamma(\cdot)$ is well defined with
$\gamma(\Lambda)\subset(0,1)$, 
$\gamma(\lambda_{-})=1$ and $\gamma(\lambda_{+})=1$ and we define
\[
r^{+}=\min\{ \gamma(s) \, \vert \, s \in \Lambda\}
\]
that belongs to $(0,1)$. One has necessarily $r^{+}=\min
R^{+}_{\alpha}$, and there is no solution of $\gamma(S_{1}^{\star})=r$
exactly when $r<r^{+}$. For $r>r^{+}$, there exist at least two
solutions by the Mean Value Theorem, and four for a subset of
$(r^{+},1)$ when $R^{+}_{\alpha}$ is not reduced to a singleton.\\

On the interval $J=(\ul S,S_{in})$, the function $\gamma(\cdot)$ is
well defined with $\gamma(J)=(0,1)$, $\gamma(\ul S)=0$ and $\gamma(S_{in})=1$. There exists at least one solution of
$f_{r}(s)=0$ on this interval.
If $R^{-}_{\alpha}=\emptyset$, $\gamma(\cdot)$ is increasing and there
exists a unique solution of $\gamma(S_{1}^{\star})=r$ on $J$ whatever
is $r$. Otherwise, $\min R^{-}_{\alpha}$ and $\max R^{-}_{\alpha}$ are
the smallest local minimum and largest local maximum of
the function $\gamma$ on the interval $J$, respectively.
Then, uniqueness of $S_{1}^{\star}$ on $J$ is achieved exactly for $r$ that
does not belong to $[\min R^{-}_{\alpha},\max R^{-}_{\alpha}]$, and for
$r \in (\min R^{-}_{\alpha},\max R^{-}_{\alpha})$, there are at least three
solutions by the Mean Value Theorem.\qed\\

\bigskip

For the proof of Theorem \ref{main_theo}, we recall below a result about asymptotically autonomous dynamics.\\

\begin{theorem}
\label{thThieme}
Let $\Phi$ be an asymptotically autonomous semi-flow with limit semi-flow $\Theta$, and let the orbit ${\mathcal O}_{\Phi}(\tau,\xi)$ have compact closure. Then the $\omega$-limit set $\omega_{\Phi}(\tau,\xi)$ is non-empty, compact, connected, invariant and chain-recurrent by the semi-flow $\Theta$ and attracts $\Phi(t,\tau,\xi)$ when $t \to \infty$.
\end{theorem}

\begin{proof}
See \cite[Theorem 1.8]{MST95}.
\end{proof}

\bigskip

We shall also need to treat a limiting case of the single chemostat
that is not covered by Proposition \ref{prop-chemostat}, when
one has exactly $\mu(S_{in})=\alpha D$ for the buffer tank with
$\mu(\cdot)$ non-monotonic, that is
provided by the following Lemma.\\

\begin{lemma}
\label{lemma}
For any $\alpha>0$ such that $ \alpha D\leq \mu(S_{in})$ and non-negative initial
condition with $X_{2}(0)>0$, the solution $S_{2}(t)$ and $X_{2}(t)$ of
(\ref{chemostat2b}) is non negative for any $t>0$ and one has
\[
\lim_{t\to+\infty}
(S_{2}(t),X_{2}(t))=(\lambda_{-}(\alpha D),S_{in}-\lambda_{-}(\alpha D))
\ .
\]
\end{lemma}

\begin{proof}
From equations (\ref{chemostat2b}) one can write the properties
\[
\begin{array}{l}
S_{2}=0 \Longrightarrow \dot S_{2} >0 \ ,\\
X_{2}=0 \Longrightarrow \dot X_{2} =0 \ ,
\end{array}
\]
and deduces that the variables $S_{2}(t)$ and $X_{2}(t)$ remain non
negative for any positive time. Considering the variable $Z_{2}=S_{2}+X_{2}-S_{in}$ whose
dynamics is $\dot Z_{2}=-\alpha D Z_{2}$, we conclude that $S_{2}(t)$ are
$X_{2}(t)$ are bounded and satisfy
\[
\lim_{t\to+\infty} S_{2}(t)+X_{2}(t) = S_{in} \ .
\]
The dynamics of the variable $S_{2}$ can thus be written as an
non autonomous scalar equation:
\[
\dot S_{2} = (\alpha D-\mu(S_{2}))(S_{in}-S_{2})-\mu(S_{2})Z_{2}(t)
\]
that is asymptotically autonomous. The study of this asymptotic
dynamics is straightforward: any trajectory that converges forwardly to the
domain $[0,S_{in}]$ has to converge to $S_{in}$ or to a zero $S_{2}^{\star}$
of $S_{2}\mapsto \alpha D -\mu(S_{2})$ on the interval $(0,S_{in})$. 
Then, the application of Theorem
\ref{thThieme} allows to conclude
that forward trajectories of the $(S_{2},X_{2})$ sub-system converge
asymptotically either to the positive steady state
$(S_{2}^{\star},S_{in}-S_{2}^{\star})$ or to the
``washout'' equilibrium $(S_{in},0)$.

For $\alpha$ such that $\alpha D < \mu(S_{in})$, there is only one
such zero, that is equal to $\lambda_{-}(\alpha D)$ (and
necessarily lower than $S_{in}$). We are in conditions of Case 3 of
Proposition \ref{prop-chemostat}: $S_{in} \in \Lambda(\alpha D)$, and the convergence to the positive equilibrium is proved.

For the limiting case $\alpha D =
\mu(S_{in})$, either $\lambda_{-}(\alpha D)=S_{in}$ when $\mu(\cdot)$
is monotonic on the interval $[0,S_{in}]$ (then the washout is the
only equilibrium), or $\lambda_{-}(\alpha
D)<S_{in}$ when $\mu(\cdot)$ is non-monotonic. In this last situation,
none of the cases of Proposition \ref{prop-chemostat} are fulfilled.
We show that for any initial
condition such that $X_{2}(0)>0$, the forward trajectory cannot converge
to the washout equilibrium. From equations (\ref{chemostat2b}) one can write
\[
X_{2}(t)=X_{2}(0)\, e^{\displaystyle
  \int_{0}^{t}(\mu(S_{2}(\tau))-\alpha D) d\tau} \ .
\]
If $X_{2}(.)$ tends to $0$, then one should have
\begin{equation}
\label{int-infty}
\int_{T}^{+\infty}(\mu(S_{2}(\tau))-\alpha D) d\tau = -\infty
\end{equation}
for any finite positive $T$. Using Taylor-Lagrange Theorem, there
exists a continuous function $\theta(.)$ in $(0,1)$ such that
\[
\mu(S_{2}(\tau))=\mu(S_{in})+\mu^{\prime}(\tilde
S_{2}(\tau))(S_{2}(\tau)-S_{in}) \mbox{ with }
\tilde S_{2}(\tau)=S_{in}+\theta(\tau)(S_{in}-S_{2}(\tau)) \ .
\]
One can then write
\[
\begin{array}{lll}
\ds \int_{T}^{+\infty}(\mu(S_{2}(\tau))-\alpha D) d\tau & = & \ds
\int_{T}^{+\infty}(\mu(S_{in})-\alpha D)d\tau
-\int_{T}^{+\infty}\mu^{\prime}(\tilde S_{2}(\tau))X_{2}(\tau)d\tau
+\int_{T}^{+\infty}\mu^{\prime}(\tilde S_{2}(\tau))Z_{2}(\tau)d\tau\\
& = & \ds -\int_{T}^{+\infty}\mu^{\prime}(\tilde
S_{2}(\tau))X_{2}(\tau)d\tau -\frac{1}{\alpha D}\int_{T}^{+\infty}\mu^{\prime}(\tilde
S_{2}(\tau))\dot Z_{2}(\tau) d\tau \ .
\end{array}
\]
Note that $S_{2}(.)$ tends to $S_{in}$ when $X_{2}(\cdot)$ tends to
$0$.
So there exists $T>0$ such that
$\tilde S_{2}(\tau)>\hat S$ for any $\tau>T$, and accordingly to Assumptions
A1, there exist positive numbers $a$, $b$ such that
$-\mu^{\prime}(\tilde S_{2}(\tau)) \in [a,b]$
for any $\tau >T$. The following inequality is obtained
\[
\int_{T}^{+\infty}(\mu(S_{2}(\tau))-\alpha D) d\tau \geq
a\int_{T}^{+\infty}X_{2}(\tau)d\tau-\frac{b}{\alpha D}|Z_{2}(T)|
\]
leading to a contradiction with (\ref{int-infty}).
\end{proof}

\bigskip

\noindent {\bf Proof of Theorem \ref{main_theo}.}

Let us consider the vector
\[
Z=\left[\begin{array}{c}
X_{1}+S_{1}-S_{in}\\
X_{2}+S_{2}-S_{in}
\end{array}\right]
\]
whose dynamics is linear:
\[
\dot Z = D\underbrace{\left(\begin{array}{cc}
\ds -\frac{1}{r} & \ds \frac{\alpha(1-r)}{r}\\[3mm]
0 & -\alpha
\end{array}\right)}_{\displaystyle A}Z \ .
\]
The matrix $A$ is clearly Hurwitz and consequently $Z$ converges
exponentially towards $0$ in forward time. Furthermore, 
variables $S_{2}$ and $X_{2}$ being non negative,
one has also from (\ref{chemostat2b}) the following properties
\[
\begin{array}{l}
S_{1}=0 \Longrightarrow \dot S_{1} \geq 0 \ ,\\
X_{1}=0 \Longrightarrow \dot X_{1} \geq 0 \ ,
\end{array}
\]
and deduces that variables $S_{1}$ and $X_{1}$ stay also non negative
in forward time.
The definition of $Z$ allows
us to conclude that variables $S_{1}$, $X_{1}$, $S_{2}$, $X_{2}$ are bounded.\\

From equations (\ref{chemostat2b}), the dynamics of the variable
$S_{1}$ can be written as an non-autonomous scalar equation:
\begin{equation}
\label{eq}
\dot S_{1} = 
\left(-\mu(S_{1})+D\frac{1-\alpha(1-r)}{r}\right)(S_{in}-S_{1})+D\frac{\alpha(1-r)}{r}(S_{2}(t)-S_{1})-\mu(S_{1})Z_{1}(t)
\ .
\end{equation}

When the initial condition of sub-system $(S_{2},X_{2})$ belongs to the
attraction basin of the washout, the dynamics (\ref{eq}) is
asymptotically autonomous with the limiting equation
\begin{equation}
\dot S_{1} = (-\mu(S_{1})+D/r)(S_{in}-S_{1}) \ .
\end{equation}
From Theorem \ref{thThieme}, we deduce that $S_{1}$ converges to
$S_{1}^{\star}$, one of the zeros of the function
\[
f(s)=(-\mu(s)+D/r)(S_{in}-s)
\]
on the interval $[0,S_{in}]$, that are $S_{in}$,
$\lambda_{-}(D/r)$ (if $\lambda_{-}(D/r)<S_{in}$) and
$\lambda_{+}(D/r)$ (if $\lambda_{+}(D/r)<S_{in}$).
The Jacobian matrix of the whole dynamics
(\ref{chemostat2b}) at steady state
$(S_{1}^{\star},S_{in}-S_{1}^{\star},S_{in},0)$
in $(Z,S_{1},S_{2})$ coordinates is
\[
\left(\begin{array}{c|c}
\\
\mbox{\Large A} & \mbox{\Large 0}\\
\\
\hline
\\
\begin{array}{cc}
-\mu(S_{1}^{\star}) & 0\\[3mm]
0 & -\mu(S_{in})
\end{array} &
\begin{array}{cc}
f^{\prime}(S_{1}^{\star}) &
\ds D\frac{\alpha(1-r)}{r}\\[3mm]
0 & \mu(S_{in})-\alpha D
\end{array}
\end{array}\right) \ .
\]
When the attraction basin of the washout of the $(S_{2},X_{2})$
subsystem is not reduced to a singleton, one has necessarily
$\mu(S_{in})<\alpha D$ (see Lemma \ref{lemma}). Furthermore, one has
$f^{\prime}(S_{in})=\mu(S_{in})-D/r$ and
$f^{\prime}(S_{1}^{\star})=-\mu^{\prime}(S_{1}^{\star})(S_{in}-S_{1}^{\star})$
when $S_{1}^{\star}< S_{in}$. So, apart two possible particular values of
$r$ that are such that $r=D/\mu(S_{in})$ or
$\lambda_{-}(D/r)=\lambda_{+}(D/r)<S_{in}$,
$f^{\prime}(S_{1}^{\star})$ is non-zero and the equilibrium is thus
hyperbolic.  Finally, we conclude about the possible asymptotic
behaviors of the whole dynamics as follows.
\begin{itemize}
\item[-] the washout equilibrium is attracting when
  $\mu(S_{in})<D/r$. When $\mu(S_{in})>D/r$, this equilibrium is a
  saddle (with a stable manifold of dimension one). Accordingly to the
Theorem of the Stable Manifold, the trajectory solution cannot
converges to such an equilibrium, excepted from a measure-zero subset of initial
conditions.
\item[-] when $\lambda_{-}(D/r)<S_{in}$, the equilibrium with
$S_{1}^{\star}=\lambda_{-}(D/r)$ is always attracting.
\item[-] when $\lambda_{+}(D/r)<S_{in}$, the equilibrium with
$S_{1}^{\star}=\lambda_{+}(D/r)$ is a saddle (with a stable manifold of dimension one). Accordingly to the
Theorem of the Stable Manifold, the trajectory solution cannot
converges to such an equilibrium, excepted from a measure-zero subset of initial
conditions.
\end{itemize}
This finishes to prove the point i. of the Theorem.\\

\bigskip

When the initial condition of sub-system $(S_{2},X_{2})$ does not belong to the
attraction basin of the washout, Proposition \ref{prop-chemostat}
ensures that $S_{2}(t)$ converges towards a positive
$S_{2}^{\star}$ that is equal to $\lambda_{-}(\alpha D)$ or $\lambda_{+}(\alpha D)$.
Then, equation (\ref{eq}) can be equivalently written as:
\begin{equation}
\label{dynaS1}
\dot S_{1} = 
(D\phi_{\alpha,r}(S_{1})-\mu(S_{1}))(S_{in}-S_{1})+D\frac{\alpha(1-r)}{r}(S_{2}(t)-S_{2}^{\star})-\mu(S_{1})Z_{1}(t) \ .
\end{equation}
So the dynamics (\ref{dynaS1}) is
asymptotically autonomous with the limiting equation
\begin{equation}
\label{reducedS1}
\dot S_{1} = (D\phi_{\alpha,r}(S_{1})-\mu(S_{1}))(S_{in}-S_{1})  \ .
\end{equation}
From Theorem \ref{thThieme}, we conclude
that forward trajectories of the $(S_{1},X_{1})$ sub-system converge
asymptotically either to a stationary point
$(S_{1}^{\star},S_{in}-S_{1}^{\star})$
where $S_{1}^{\star}$ is a zero of the function 
\[
f_{r}(s)=D\phi_{\alpha,r}(s)-\mu(s)
\]
on the interval $(0,S_{in})$, either to the washout point $(S_{in},0)$. 
We show that this last case is not possible.
From equations (\ref{chemostat2b}), one has
\[
X_{1}=0 \Longrightarrow \dot X_{1} =D\frac{\alpha(1-r)}{r}X_{2}
\]
and as $X_{2}(t)$ converges to a
positive value, we deduce that $X_{1}(t)$ cannot converges towards
$0$.\\

The functions $f_{r}$ being analytic for any $r$, the roots $S_{1}^{\star}$ are
isolated. As for the proof of Proposition \ref{prop1} we consider the function
\[
\gamma(s)=\frac{\ul S -s}{\ul S - S_{in}+(S_{in}-s)\mu(s)/D}
\]
that is analytic on its domain of definition and such that
\[
f_{r}(s)=0 \Longleftrightarrow \gamma(s)=r \ .
\]
This shows that, excepted for some isolated values of $r$ in $(0,1)$, the
zero of $f_{r}$ are such that $f_{r}^{\prime}(S_{1}^{\star})\neq 0$.\\

Let us now write the Jacobian matrix $J^{\star}$ of dynamics
(\ref{chemostat2b}) at steady state
$E^{\star}=(S_{1}^{\star},S_{in}-S_{1}^{\star},S_{2}^{\star},S_{in}-S_{2}^{\star})$
in $(Z,S_{1},S_{2})$ coordinates:
\[
J^{^{\star}}=\left(\begin{array}{c|c}
\\
\mbox{\Large A} & \mbox{\Large 0}\\
\\
\hline
\\
\begin{array}{cc}
-\mu(S_{1}^{\star}) & 0\\[3mm]
0 & -\mu(S_{2}^{\star})
\end{array} &
\begin{array}{cc}
f_{r}^{\prime}(S_{1}^{\star})(S_{in}-S_{1}^{\star}) &
\ds D\frac{\alpha(1-r)}{r}\\[3mm]
0 & -\mu^{\prime}(S_{2}^{\star})(S_{in}-S_{2}^{\star})
\end{array}
\end{array}\right) \ .
\]
Considering the following facts:\\
\indent i. $A$ is Hurwitz,\\
\indent ii. $\Lambda(\alpha D)\neq\emptyset$ implies that
$S_{2}^{\star}$ is not equal to $\hat S$. So one has
$\mu^{\prime}(S_{2}^{\star})\neq 0$ (cf Assumptions A1),\\
\indent iii. $f_{r}^{\prime}(S_{1}^{\star})\neq 0$ for almost any $r$,\\
we conclude that any equilibrium $E^{\star}$ is hyperbolic
(for almost any $r$) and is\\
\indent- a saddle point when $\mu^{\prime}(S_{2}^{\star})>0$ or
$f_{r}^{\prime}(S_{1}^{\star})>0$,\\
\indent- an exponentially stable critical point otherwise.\\

Furthermore, the left endpoints of the connected components of the set
$\Gamma_{\alpha,r}(D)$ are exactly the roots of $f_{r}$ with
$f_{r}(S_{1}^{\star})<0$.
Finally, from the Stable Manifold Theorem we conclude that, excepted from the stable
manifolds of the saddle equilibria, the trajectory converges to an
equilibrium that is such that $S_{2}^{\star}=\lambda_{-}(\alpha D)$
and $f_{r}(S_{1}^{\star})<0$. This ends the proof of point ii.
\qed\\

\bigskip

\begin{proposition}
\label{prop5}
Assume that the hypotheses A1 are fulfilled with $\Lambda(D)\neq\emptyset$
and $\lambda_{+}(D)<S_{in}$.
There exist buffered configurations
with an additional tank of volume $V_{2}$ that possesses a unique globally
exponentially stable positive equilibrium from any initial condition with $S_{2}(0)>0$, exactly when $V_{2}$
fulfills the condition
\begin{equation}
\label{condV2}
\frac{V_{2}}{V} >  \left(\frac{V_{2}}{V}\right)_{\inf}=\frac{\ds \max_{s \in
    (\lambda_{+}(D),S_{in})}\varphi(s)}{\ds \max_{s \in [0,\bar s]}\psi(s)} \ ,
\end{equation}
where the functions $\varphi(\cdot)$ and $\psi(\cdot)$ are defined as
follows:
\begin{equation}
\label{defvarphi_and_psi}
\varphi(s)=(S_{in}-s)(D-\mu(s)) \ , \qquad \psi(s)=\mu(s)(S_{in}-s) \ ,
\end{equation} 
and $\bar s$ is the number
\begin{equation}
\label{bar_s}
\bar s = \lim_{\alpha\to \mu(S_{in})}S_{2}^{\star}(\alpha) \ .
\end{equation}
The dilution rate $D_{2} \in (0,\mu(S_{in}))$ has then to satisfy the condition
\[
\max_{s \in (\lambda_{+}(D),S_{in})}\varphi(s) < 
D_{2}\frac{V_{2}}{V}(S_{in}-S_{2}^{\star}(D_{2})) < S_{in} \ .
\]
Furthermore, one has
\begin{equation}
\label{ineqDeltaV}
\left(\frac{V_{2}}{V}\right)_{\inf} < \left(\frac{\Delta
    V}{V}\right)_{\inf} \ .
\end{equation}
\end{proposition}

\bigskip

\noindent {\bf Proof of Proposition \ref{prop5}.}
One can straightforwardly
check on equations (\ref{chemostat2}) that a positive equilibrium in the first
tank has to fulfill
\begin{equation}
\label{eq-varphi}
\varphi(S_{1}^{\star})=D_{2}\frac{V_{2}}{V}(S_{in}-S_{2}^{\star}(D_{2}))
\ .
\end{equation}
Let us examine some properties of the function $\varphi$ on the interval $(0,S_{in})$:
\begin{itemize}
\item[.]  $\varphi$ is negative exactly on the interval $\Lambda(D)$,
\item[.] $\varphi^{\prime}$ is negative on $(0,\lambda_{-}(D))$ with
  $\varphi(0)=S_{in}$ and $\varphi(\lambda_{-}(D))=0$,
\item[.] $\varphi(\lambda_{+}(D))=\varphi(S_{in})=0$ and $\varphi$ reaches
  its maximum $m^{+}$ on the sub-interval $(\lambda_{+}(D),S_{in})$,
  that is strictly less than $S_{in}=\varphi(0)$,
\end{itemize}
from which we deduce that there exists
a unique solution of $\varphi(s)=c$ on the whole interval
$(0,S_{in})$ exactly when $c\in (m^{+},S_{in})$ (see Figure
(\ref{figvarphi}) as an illustration).
\begin{figure}[ht]
\begin{center}
\includegraphics[width=6cm]{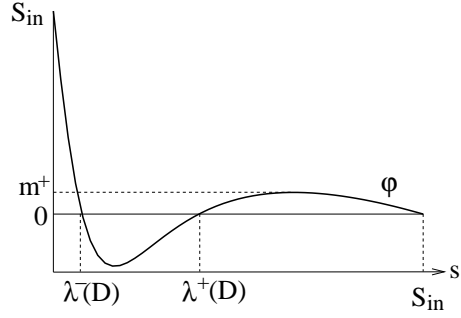}
\caption{Illustration of the graph of the function $\varphi$\label{figvarphi}}
\end{center}
\end{figure}
The configurations for which there exists a unique
$S_{1}^{\star} \in (0,S_{in})$ solution of the equation
(\ref{eq-varphi}) are exactly those that fulfill the condition
$D_{2}\frac{V_{2}}{V}(S_{in}-S_{2}^{\star}(D_{2}))\in (m^{+},S_{in})$, or equivalently
\[
\frac{m^{+}}{D_{2}(S_{in}-S^{\star}_{2}(D_{2}))} < \frac{V_{2}}{V} < 
\frac{S_{in}}{D_{2}(S_{in}-S^{\star}_{2}(D_{2}))}
\]
with $D_{2} \in (0,\mu(S_{in}))$.
Then, Theorem \ref{main_theo} with $\alpha=D_{2}/D$ and $r=1/(1+\frac{V_{2}}{V})$
guarantees that the unique positive equilibrium 
$(S^{\star}_{1},S_{in}-S^{\star}_{1},S^{\star}_{2}(D),S_{in}-S^{\star}_{2}(D))$
is globally exponentially stable on the domain
$\Rset_{+}^{2}\times\Rset_{+}^{\star}\times\Rset_{+}$.\\

Among all such configurations, the infimum of $V_{2}/V$ can be
approached arbitrarily close when $D_{2}$ is maximizing the function
\[
D_{2} \mapsto \alpha(S_{in}-S_{2}^{\star}(D_{2}))
\]
on $[0,\mu(S_{in})]$, that exactly amounts to maximize the function
$\psi(\cdot)$ on the interval $[0,\bar s]$.

Finally, let $s^{\star}$ be a minimizer of $\varphi$ on
$(\lambda_{+}(D),S_{in})$. One has
$\mu(s^{\star})>\mu(S_{in})=\mu(\bar s)$ and can write
\[
\left(\frac{V_{2}}{V}\right)_{\inf} \leq
\frac{\varphi(s^{\star})}{\psi(\bar s)}
<\frac{(S_{in}-s^{\star})(D-\mu(S_{in}))}{\mu(S_{in})(S_{in}-\bar
  s)}
=\frac{S_{in}-s^{\star}}{S_{in}-\bar
  s}\left(\frac{\Delta V}{V}\right)_{\inf}
\]
which leads to the inequality (\ref{ineqDeltaV}). \qed


\begin{thebibliography}{99}

\bibitem{AN01}
{\sc P. Amarasekare and R. Nisbet},
{\em Spatial heterogeneity, source‐sink dynamics, and the local
  coexistence of competing species}, The American Naturalist, 158(6)
(2001), 572--584.

\bibitem{A68}
{\sc J.F. Andrews},
{\em A mathematical model for the continuous culture of microorganisms
utilizing inhibitory substrates},
Biotech. Bioengrg., 10 (1968), 707--723.


\bibitem{BW85}
{\sc G. J. Butler and G. S. K. Wolkowicz},
{\em A mathematical model of the chemostat with a general class of
functions describing nutrient uptake},
SIAM J.~Appl.~Math.  45 (1985), 138--151.

\bibitem{BC76}
{\sc A. Bush and A. Cook}
{\em The effect of time delay and growth rate inhibition in the bacterial treatment of wastewater}
J. Theor Biol. 63(2) (1976), 385--395.


\bibitem{DBBT96}
{\sc C. de Gooijer, W. Bakker, H. Beeftink and J. Tramper},
{\em Bioreactors in series: An overview of design
procedures and practical applications}, Enzyme and Microbial Technology,  18 (1996), 202--219.

\bibitem{DGC02}
{\sc E. Di Mattia, S. Grego and I. Cacciari},
{\em Eco-physiological characterization of soil bacterial populations
  in different states of growth}
Microb. Ecol. 43(1) (2002),  34--43.

\bibitem{DB84}
{\sc D. Dochain and G. Bastin}
{\em Adaptive identification and control algorithms for non linear
  bacterial growth systems.}
Automatica, 20(5) (1984), 621--634.

\bibitem{DV01}
{\sc D. Dochain and P. Vanrolleghem},
{\em Dynamical Modelling and Estimation in Wastewater treatment
  Processes},
IWA Publishing, U.K. (2001).

\bibitem{DHRL06}
{\sc A. Dram\'e, J. Harmand, A. Rapaport and C. Lobry},
{\em Multiple steady state
profiles in interconnected biological systems}, Mathematical  and  Computer
Modelling of Dynamical Systems,  12 (2006), 379--393.

\bibitem{EE90}
{\sc H. El-Owaidy and O. El-Leithy},
{\em Theoretical studies on extinction in the gradostat}
Mathematical Biosciences,  101(1) (1990), 1--26.

\bibitem{FS81}
{\sc A. Fredrickson and G. Stephanopoulos},
{\em Microbial Competition}
Science, 213 (1981), 972--979.

\bibitem{FW86}
{\sc H. Freedman and G. Wolkowicz},
{\em Predator-prey systems with group defence: The paradox of
  enrichment revisited.}
Bulletin of Mathematical Biology, 48(5/6) (1986) 493--508.

\bibitem{FHN11}
{\sc C. Fritzsche, K. Huckfeldt and E.-G. Niemann},
{\em Ecophysiology of associative nitrogen fixation in a rhizosphere
  model in pure and mixed culture},
FEMS Microbiology Ecology, 8(4) (2011), 279--290.


\bibitem{GTVP09}
{\sc A. Gaki, Al. Theodorou, D. Vayenas and S. Pavlou},
{\em Complex dynamics of microbial competition in the gradostat},
Journal of Biotechnology,  139(1) (2009) pp 38--46.

\bibitem{GGLM10}
{\sc D. Gravel, F. Guichard, M. Loreau and N. Mouquet},
{\em Source and sink dynamics in metaecosystems.}
Ecology, 91 (2010), 2172--2184.

\bibitem{HRG11}
{\sc I. Haidar, A. Rapaport and F. G\'erard},
{\em Effects of spatial structure and diffusion on the performances
  of the chemostat},
Mathematical Biosciences and Engineering,  8(4) (2011), 953--971.

\bibitem{HRM06}
{\sc J. Harmand, A. Rapaport and F. Mazenc},
{\em Output tracking of continuous bioreactors through recirculation
  and by-pass}, 
Automatica, 42(7) (2006) 1025--1032.  


\bibitem{HRT99}
{\sc J. Harmand, A. Rapaport and A. Trofino},
{\em Optimal design of two interconnected bioreactors--some new results}, American Institute of Chemical Engineering Journal,
 49 (1999), 1433--1450.

\bibitem{HJ54}
{\sc A. Hasler and W. Johnson},
{\em The in situ chemostat -- a self-contained continuous culturing
  and water sampling system.}
Limnol. Oceanogr. 79 (1954), 326--331.

\bibitem{HS94}
{\sc J. Hofbauer and W. So},
{\em Competition in the gradostat: the global stability problem
  Original Research}
Nonlinear Analysis: Theory, Methods \& Applications,  22(8)
(1994), 1017--1031.

\bibitem{HYSS98}
{\sc Y. Higashi, N. Ytow, H. Saida and H. Seki},
{\em In situ gradostat for the study of natural phytoplankton
  community with an experimental nutrient gradient}
Environmental Pollution, 99 (1998), 395--404.

\bibitem{HR89}
{\sc G. Hill and C. Robinson},
{\em Minimum tank volumes for CFST bioreactors in series},
The Canadian Journal of Chemical Engineering,  67 (1989), 818--824.

\bibitem{JSTW87}
{\sc W. Jaeger, J.-H. So, B. Tang and P. Waltman}, 
{\em Competition in the gradostat}, 
J. Math. Biol.  25 (1987), 23--42.

\bibitem{JM74}
{\sc H. Jannash and R. Mateles},
{\em Experimental bacterial ecology studies in continuous culture},
Advanced in Microbial Physiology 11 (1974), 165--212.

\bibitem{LTVP98}
{\sc P. Lenas, N. Thomopoulos, D. Vayenas and S. Pavlou},
{\em Oscillations of two competing microbial populations in configurations of two interconnected chemostats},
Mathematical Biosciences,  148(1) (1998), 43--63.

\bibitem{L74}
{\sc S. Levin},
{\em Dispersion and population interactions},
The American Naturalist, 108(960) (1974), 207--228.

\bibitem{L77}
{\sc J. La Rivi\`ere},
{\em Microbial ecology of liquid waste},
Advances in Microbial Ecology, 1 (1977), 215--259.

\bibitem{L98}
{\sc B. Li},
{\em Global asymptotic behavior of the chemostat: General response functions and
differential removal rates},
SIAM J.~Appl.~Math.  59 (1998), 411--22.

\bibitem{L10}
{\sc M. Loreau},
{\em From Populations to Ecosystems: Theoretical Foundations for a New
Ecological Synthesis.}
Princeton University Press, Princeton (2010).

\bibitem{LDGGGLLMM13}
{\sc M. Loreau, T. Daufresne, A. Gonzalez, D. Gravel, F. Guichard,
  S.J. Leroux, N. Loeuille, F. Massol and N. Mouquet.}
{\em Unifying sources and sinks in ecology and Earth sciences.} 
Biological Review 88 (2013), 365--79.

\bibitem{LW81}
{\sc R. Lovitt and J. Wimpenny},
{\em The gradostat: A bidirectional compound chemostat
and its applications in microbial research},
Journal of General Microbiology,  127 (1981), 261--268.

\bibitem{LT82}
{\sc K. Luyben and J. Tramper},
{\em Optimal design for continuously stirred tank reactors in series using Michaelis-Menten kinetics}, Biotechnology and
Bioengineering,  24 (1982), 1217--1220.

\bibitem{MW67}
{\sc R. MacArthur and E. Wilson},
{\em The Theory of Island Biogeography},
Princeton University Press (1967).

\bibitem{MST95}
{\sc M. Mischaikow, H. Smith and H. Thieme},
{\em  Asymptotically autonomous semiflows: chain recurrence and Lyapunov functions},
Transactions of the American Mathematical Society,  347(5)
(1995), 1669--1685.

\bibitem{M50}
{\sc J. Monod},
{\em La technique de la culture continue: Th\'eorie et applications},
Annales de l'Institut Pasteur, 79 (1950), 390--410.

\bibitem{NT06}
{\sc S. Nakaoka and Y. Takeuchi},
{\em Competition in chemostat-type equations with two habitats},
Mathematical Bioscience,  201 (2006), 157--171.

\bibitem{NS06}
{\sc M. Nelson and H. Sidhu},
{\em Evaluating the performance of a cascade of two bioreactors},
Chemical Engineering Science,  61 (2006), 3159--3166.

\bibitem{NS50}
{\sc A. Novick and L. Szilard},
{\em Description of the chemostat},
Science, 112 (1950), 715--716.

\bibitem{P75}
{\sc J. Pirt},
{\em Principles of Microbe and Cell Cultivation},
Blackwell Scientific Publications (1975).

\bibitem{RH08}
{\sc A. Rapaport and J. Harmand},
{\em Biological control of the chemostat with non-monotonic response
  and different removal rates},
Mathematical Biosciences and Engineering,  5(3) (2008), 539--547.  

\bibitem{RHM08}
{\sc A. Rapaport, J. Harmand and F. Mazenc},
{\em Coexistence in the design of a series of two chemostats},
Nonlinear Analysis, Real World Applications,
 9 (2008), 1052--1067.

\bibitem{SAL12}
{\sc A. Schaum, J. Alvarez and T. Lopez-Arenas},
{\em Saturated PI control of continuous bioreactors with Haldane
  kinetics}
Chem. Eng. Science, 68 (2012), 520--529.

\bibitem{ST89}
{\sc H. Smith and B. Tang},
{\em Competition in the gradostat: the role of the communication rate},
J. Math. Biol. 27(2) (1989), 139--165.

\bibitem{STW91}
{\sc H. Smith, B. Tang and P. Waltman}, 
{\em Competition in an n-vessel gradostat}, 
SIAM J. Appl. Math.  51 (1991), 1451--1471.

\bibitem{SW91}
{\sc H. Smith and P. Waltman}, 
H.L. Smith, P. Waltman, 
{\em The gradostat: a model of competition along a nutrient gradient}, 
J. Microb. Ecol.  22 (1991), 207--226.

\bibitem{SW95}
{\sc H. Smith and P. Waltman}, 
{\em The theory of chemostat, dynamics of microbial competition},
Cambridge Studies in Mathematical Biology, Cambridge University Press (1995).

\bibitem{SW00}
{\sc H. Smith and P. Waltman}, 
{\em Competition in the periodic gradostat},
Nonlinear Analysis: Real World Applications,  1(1) (2000), 177--188.

\bibitem{SF79}
{\sc G. Stephanopoulos and A. Fredrickson},
{\em Effect of inhomogeneities on the coexistence of competing microbial populations}, Biotechnology and Bioengineering,  21 (1979), 1491--1498.

\bibitem{T86}
{\sc B. Tang}, 
{\em Mathematical investigations of growth of microorganisms in the gradostat},
J. Math. Biol. 23 (1986), 319--339.

\bibitem{T94}
{\sc B. Tang}, 
{\em Competition models in the gradostat with general nutrient uptake
  functions}, 
Rocky Mountain J. Math.  24(1) (1994), 335--349.

\bibitem{V77}
{\sc H. Veldcamp},
{\em Ecological studies with the chemostat},
Advances in Microbial Ecology, 1 (1977), 59--95.

\bibitem{WL92}
{\sc G. Wolkowicz and Z. Lu},
{\em Global dynamics of a mathematical model of competition in the chemostat: general response functions and differential death rates},
SIAM  J.~Appl.~Math.  52 (1992), 222--233.

\bibitem{XR01}
{\sc D. Xiao and S. Ruan},
{\em Global analysis in a predator-prey system with nonmonotonic functional response.}
SIAM Journal on Applied Mathematics, 61(4) (2001), 11445--72.


\bibitem{Z92}
{\sc A. Zaghrout},
{\em Asymptotic behavior of solutions of competition in gradostat with two limiting complementary substrates}
Applied Mathematics and Computation, 49 (1) (1992), 19--37.


\end{thebibliography}
\end{document}